\documentclass[10pt,a4paper,reqno]{amsart}
\title[]{On the supercritical defocusing NLW outside a ball}

\usepackage{%
    amssymb%
    ,silence
    ,graphicx%
    ,mathrsfs%
    ,eucal%
    ,xcolor%
    ,enumitem%
    ,geometry%
    ,esint}
\usepackage[%
    pdfauthor={Piero D'Ancona}%
    ,colorlinks=true%
    ,linkcolor=magenta%
    ,citecolor=cyan%
    ,urlcolor=blue%
    ,hyperfootnotes=false%
]{hyperref}

\numberwithin{equation}{section}
\newtheorem{theorem}{Theorem}[section]

\newtheorem{lemma}[theorem]{Lemma}

\newtheorem{proposition}[theorem]{Proposition}
\theoremstyle{remark}
\newtheorem{remark}{Remark}[section]

\theoremstyle{definition}
\newtheorem{definition}[theorem]{Definition}
\usepackage{stmaryrd}
  {\baselineskip0pt\normalfont\footnotesize
  \setlength\abovedisplayskip{1pt}
  \setlength\abovedisplayshortskip{1pt}
  \setlength\belowdisplayskip{1pt}
  \setlength\belowdisplayshortskip{1pt}
  \vskip2pt\par\noindent$\llbracket$\ignorespaces}%
  {\ignorespaces$\rrbracket$\vskip2pt}

\newcommand{\bra}[1]{\langle #1 \rangle}
\newcommand{\one}[1]{\mathbf{1}_{#1}}

\date{\today}

\author[P.~D'Ancona]{Piero D'Ancona}
\address{Piero D'Ancona: 
Dipartimento di Matematica\\
Sapienza Universit\`{a} di Roma\\
Piazzale A.~Moro 2\\
00185 Roma\\
Italy}
\email{dancona@mat.uniroma1.it}

\thanks{%
The author is partially supported by the Project Ricerca Scientifica Sapienza 2016: ``Metodi di Analisi Reale e Armonica per problemi stazionari ed evolutivi''.
}

\begin{document}

\begin{abstract}
  We study a defocusing semilinear wave equation, with a power nonlinearity $|u|^{p-1}u$, defined outside the unit ball of $\mathbb{R}^{n}$, $n\ge3$, with Dirichlet boundary conditions. We prove that if $p>n+3$ and the initial data are nonradial perturbations of large radial data, there exists a global smooth solution. The solution is unique among energy class solutions satisfying an energy inequality. The main tools used are the Penrose transform and a Strichartz estimate for the exterior linear wave equation perturbed with a large, time dependent potential.
\end{abstract}

\maketitle



\section{Introduction}\label{sec:intr}

Consider the Cauchy problem for the defocusing wave equation on 
$\mathbb{R}_{t}\times \mathbb{R}^{n}_{x}$:
\begin{equation}\label{eq:WEgen}
  \square u+|u|^{p-1}u=0,
  \qquad
  u(0,x)=u_{0}(x),
  \qquad
  u_{t}(0,x)=u_{1}(x)
\end{equation}
with Sobolev initial data $(u_{0},u_{2})\in H^{s}\times H^{s-1}$.
The existence of global solutions to this problem has 
been explored in considerable detail. 
The \emph{critical power} for global smooth solvability is
$p_{cr}(n)=1+\frac{4}{n-2}$ for $n\ge3$, 
while $p_{cr}(1)=p_{cr}(2)=\infty$.
Global existence in the energy space is known for
$p\le p_{cr}(n)$
\cite{GinibreVelo94-a},
\cite{ShatahStruwe94-a},
\cite{Kapitanski94};
regularity for $p=p_{cr}$
has been explicitly proved up to dimension 7
\cite{ShatahStruwe93-a},
\cite{Kapitanski94} and should
hold for all dimensions.
For \emph{supercritical} nonlinearities $p>p_{cr}(n)$,
very little is known, and the question of global existence 
or blow up is still open.

If we restrict to spherically symmetric solutions, the
difficulty of the problem is essentially the same.
However, if a radial solution blows up, 
the blow up must occur at the origin.
This follows at once from the Sobolev embedding for radial 
functions
\begin{equation}\label{eq:radialstrauss}
  |x|^{\frac{n}{2}-1}|u(x)|\lesssim\|\partial u\|_{L^{2}},
\end{equation}
which is a well known special
case of the family of inequalities
\begin{equation}\label{eq:nonradialstr}
  |x|^{\frac np-\sigma}|u(x)|
  \lesssim
  \||D|^{\sigma}u\|_{L^{p}_{|x|}L^{r}_{\omega}},
  \qquad
  \frac{n-1}{r}+\frac1p<\sigma<\frac np
\end{equation}
(see \cite{DAnconaLuca12-a}). Here the norm
$L^{p}_{|x|}L^{r}_{\omega}$ is an $L^{p}$ norm
in the radial direction of the $L^{r}_{\omega}$
norm in the angular direction. We give a precise
statement of this blow--up alternative result
including its short proof:

\begin{proposition}[]\label{pro:blowupat0}
  Let $u_{0},u_{1}$ be spherically symmetric test functions, 
  and $T$ the supremum of times $\tau>0$ such that
  a smooth solution $u(t,x)$ of \eqref{eq:WEgen} exists on
  $[0,\tau]\times \mathbb{R}^{n}$.
  Assume $T<\infty$.
  Then $u$ blows up at $x=0$ as $t\to T$,
  in the sense that for any $R>0$,
  $u$ is unbounded on the cylinder
  $\{|x|<R,\ t<T\}$.
\end{proposition}

\begin{proof}
  By standard local existence results
  (e.g.~\cite{ShatahStruwe98-a}, Chapter 5) we know that
  $T>0$, and the
  solution is spherically symmetric for all $t\in[0,T)$
  by uniqueness of smooth solutions.
  Denote the energy of the solution $u$ at time $t$ by
  \begin{equation*}
    \textstyle
    E(t):=
    \frac12
    \|\partial_{t,x}u(t)\|_{L^{2}(\mathbb{R}^{n})}^{2}
    +
    \frac{1}{p+1}
    \int|u(t)|^{p+1}dx.
  \end{equation*}
  By conservation of energy and the radial estimate 
  \eqref{eq:radialstrauss} we have
  \begin{equation*}
    \textstyle
    |x|^{n-2}|u(t,x)|^{2}
    \le
    CE(t)=CE(0),
    \qquad
    0\le t<T.
  \end{equation*}
  Assume by contradiction that
  $|u(t,x)|\le M$ for $(t,x)$ in some cylinder
  $\{|x|<R,\ t<T\}$, then by the previous bound we get
  \begin{equation*}
    |u(t,x)|\le M+(CE(0))^{1/2}R^{1-n/2}<\infty
    \quad\text{for}\quad 
    0\le t<T,\quad x\in \mathbb{R}^{n}
  \end{equation*}
  and by standard arguments we
  can extend $u$ to a strip $[0,T']\times \mathbb{R}^{n}$
  for some $T'>T$, against the assumptions.
\end{proof}

In order to obtain a global radial solution, it would be
sufficient to prevent blow-up at 0. This is the case
if we replace the whole space with the exterior of a ball:
\begin{equation*}
  \Omega=\{x\in \mathbb{R}^{n}:|x|>1\}
\end{equation*}
and consider the mixed problem on $\mathbb{R}^{+}\times \Omega$
with Dirichlet boundary conditions
\begin{equation}\label{eq:pbm}
  \square u+|u|^{p-1}u=0,
  \qquad
  u(0,x)=u_{0}(x),
  \qquad
  u_{t}(0,x)=u_{1}(x),
  \qquad
  u(t,\cdot)\vert_{\partial \Omega}=0.
\end{equation}
In this situation the local solution can not blow up, due to the 
conservation of $H^{1}$ energy and the radial estimate
\eqref{eq:radialstrauss}. This implies the existence of a global
radial solution for arbitrary powers $p>1$.
Indeed, we have the following
result, where we use the notation
\begin{equation*}
  C^{k}X=C^{k}(\mathbb{R}^{+};X),
  \qquad
  X=L^{2}(\mathbb{R}^{n})
  \ \text{or}\ 
  H^{s}(\mathbb{R}^{n})
  \ \text{or}\ 
  H^{1}_{0}(\Omega).
\end{equation*}

\begin{proposition}[]\label{pro:radial}
  Let $\Omega=\mathbb{R}^{n}\setminus \overline{B(0,1)}$, 
  $n\ge2$, $p>1$ and let
  $u_{0}\in H^{1}_{0}(\Omega)\cap H^{2}(\Omega)
    \cap L^{p+1}(\Omega)$, 
  $u_{1}\in H^{1}_{0}(\Omega)$
  be two radially symmetric functions.

  Then the mixed problem \eqref{eq:pbm}
  has a global solution 
  $u\in C^{2}L^{2}\cap C^{1}H^{1}_{0}\cap CH^{2}$,
  satisfying the conservation of energy
  \begin{equation}\label{eq:consE}
    \textstyle
    E(u(t))=E(u(0)),
    \qquad
    E(u(t)):=
    \frac12
    \|\partial_{t,x}u(t)\|_{L^{2}(\Omega)}^{2}
    +
    \frac{1}{p+1}
    \int_{\Omega}|u(t)|^{p+1}dx
  \end{equation}
  and the uniform bound
  \begin{equation}\label{eq:unibound}
    \|u\|_{L^{\infty}(\mathbb{R}^{+}\times \Omega)}
    \lesssim
    C(\|u_{0}\|_{H^{1}},\|u_{1}\|_{L^{2}}).
  \end{equation}
  If $v\in C^{2}L^{2}\cap C^{1}H^{1}_{0}\cap CH^{2}$ is
  a second solution of \eqref{eq:pbm}
  with the same data
  which is either radially symmetric in $x$,
  or locally bounded on $\mathbb{R}^{+}\times \Omega$, then
  $v \equiv u$.
\end{proposition}

The existence of large radial solutions suggests
naturally the question of stability: do nonradial perturbations
of radial data give rise to global solutions?

The uniform bound \eqref{eq:unibound} is not
sufficient for a perturbation argument. 
However, we can prove
an actual decay estimate, obtained by reduction to a mixed 
problem with moving boundary on $\mathbb{S}^{n}$ via the
Penrose transform. Define for $M>1$ the quantity
\begin{equation}\label{eq:weighsob}
  \textstyle
  C_{M}(u_{0},u_{1}):=
  \|u_{0}\|_{H^{N_{0}+1,N_{0}}(|x|\ge M)}
  +
  \|u_{1}\|_{H^{N_{0},N_{0}+1}(|x|\ge M)},
  \qquad
  N_{0}=\lfloor \frac n2 \rfloor+1.
\end{equation}
Then we have

\begin{theorem}[Decay of the radial solution]\label{the:decayrad}
  Let $u$ be the solution constructed in 
  Proposition \ref{pro:radial}. Assume that for some $M>1$
  the data satisfy
  \begin{equation}\label{eq:defX}
    \|(u_{0},u_{1})\|_{M}:=C_{M}(u_{0},u_{1})+
    \|u_{0}\|_{H^{2}}+\|u_{1}\|_{H^{1}}+
    \|\bra{x}^{n \frac{p-1}{p+1}-1}u_{0}\|_{L^{p+1}}
    <\infty.
  \end{equation}
  Then the following decay estimate holds
  \begin{equation}\label{eq:decaysol}
    |u(t,x)|\le C\bigl(\|(u_{0},u_{1})\|_{M}\bigr)\cdot
    |x|^{1-\frac n2}\langle t+|x| \rangle^{-\frac 12}
    \langle t-|x| \rangle^{-\frac 12}.
  \end{equation}
\end{theorem}

If the initial data are smoother then regularity propagates,
and in addition higher Sobolev norms remain
bounded as $t\to \infty$. In order to state the regularity result,
we introduce briefly the \emph{nonlinear compatibility conditions},
discussed in greater detail in 
Section \ref{sec:globalrad} (see Definition \ref{def:nlcc} and also
Definition \ref{def:lcc}). For the mixed problem on
$\mathbb{R}^{+}\times \Omega$
\begin{equation*}
  u_{tt}=\Delta u+f(u),
  \qquad
  u(0,x)=u_{0},
  \qquad
  u_{t}(0,x)=u_{1},
  \qquad
  u(t,\cdot)\vert_{\partial \Omega}=0
\end{equation*}
we define formally the sequence of functions $\psi_{j}$ as follows:
\begin{equation*}
  \psi_{0}=u(0,x)=u_{0},
  \quad
  \psi_{1}=u_{t}(0,x)=u_{1},
  \quad
  \psi_{j}=\partial_{t}^{j}u(0,x)=
  \Delta \psi_{j-2}+\partial_{t}^{j-2}(f(u))\vert_{t=0},
\end{equation*}
where in the definition of $\psi_{j}$ we set recursively
$\partial_{t}^{k}u(0,x)=\psi_{k}(x)$ for $0\le k\le j-2$.
Then we say that the data $(u_{0},u_{1},f(s))$ satisfy the
\emph{nonlinear compatibility conditions} of order $N\ge1$ if
\begin{equation}\label{eq:ccond}
  \psi_{j}\in H^{1}_{0}(\Omega)
  \quad\text{for}\quad 
  0\le j\le N.
\end{equation}
Note that if $f(s)$ vanishes at $s=0$ of sufficient order,
to satisfy the nonlinear compatibility conditions
it is sufficient to assume that the initial data
$u_{0},u_{1}$ belong to $H^{k}_{0}(\Omega)$
for $k$ large enough.

We can now state the higher regularity result for the
radial solutions:

\begin{theorem}[]\label{the:regul}
  Let  $n\ge3$, $p>N\ge1$, $p> \frac{2n}{n-2}$.
  Assume the radially symmetric data
  $u_{0}\in H^{N+1}(\Omega)$, $u_{1}\in H^{N}$  and
  $f(u)=|u|^{p-1}u$ satisfy the nonlinear compatibility conditions
  of order $N$ and condition \eqref{eq:defX}.

  Then the radial solution constructed in 
  Proposition \ref{pro:radial} belongs to
  $\in C ^{k}H^{N+1-k}\cap C^NH^{1}_{0}$
  for all $ 0\le k\le N+1$
  and satisfies (besides \eqref{eq:consE})
  the uniform bounds
  \begin{equation*}
    \|u\|_{L^{\infty}L^{2}}\le
    C(\|(u_{0},u_{1})\|_{M}),
  \end{equation*}
  \begin{equation}\label{eq:boundsHk}
    \|\nabla_{t,x}u\|_{Y^{\infty,2;k}}
    \le
    C
    \Bigl(\|(u_{0},u_{1})\|_{M},
      \|u_{0}\|_{H^{k}},\|u_{1}\|_{H^{k-1}}
    \Bigr),
    \qquad
    1\le k< p-1.
  \end{equation}
\end{theorem}

The main result of the paper is the following:

\begin{theorem}[]\label{the:quasiradial}
  Let $n\ge3$, $N\ge \frac 52n$,
  $p> n+3$, and $u_{0},u_{1}$ radial
  functions satisfying the assumptions in Theorem \ref{the:regul}.
  Then there exists 
  $\epsilon=\epsilon(u_{0},u_{1})>0$ such that the following holds. 

  Assume
  $v_{0}\in H^{N+1}(\Omega)$, 
  $v_{1}\in H^{N}(\Omega)$ and $f(u)=|u|^{p-1}u$
  satisfy the
  nonlinear compatibility conditions of order $N$
  and $\|u_{0}-v_{0}\|_{H^{N+1}}< \epsilon$,
  $\|u_{1}-v_{1}\|_{H^{N}}< \epsilon$.
  Then Problem \eqref{eq:pbm} with data $v_{0},v_{1}$ has a global
  solution $v\in C^{k}H^{N-k}\cap C^{N-1_{b}}H^{1}_{0}$,
  $0\le k\le N+1$.
\end{theorem}

In other words,
nonradial perturbations of large radial initial data in a high 
Sobolev norm do
give rise to global smooth quasiradial solutions. 
Note that these solutions are
unique in the class of locally bounded global solutions, but 
one can not in principle exclude the existence of other energy class solutions. However, one can prove a weak--strong
uniqueness result, which
implies in particular that the solution constructed
in Theorem \ref{the:quasiradial} is the unique energy class
solution satisfying an energy inequality:

\begin{theorem}[]\label{the:wsuniq}
  Suppose all the assumptions in Theorem \ref{the:quasiradial}
  are satisfied. Let $I$ be an open interval containing
  $[0,T]$ and
  $\underline{v}\in C(I;H^{1}_{0})\cap C^{1}(I;L^{2})
    \cap L^{\infty}(I;L^{p+1}(\Omega))$
  a distributional solution for $0\le t\le T$
  to Problem \eqref{eq:pbm} with the same initial data
  as $v$,
  which satisfies an energy inequality 
  $E(\underline{v}(t))\le E(\underline{v}(0))$
  (see \eqref{eq:consE}). Then we have $\underline{v}(t)=v(t)$
  for $0\le t\le T$.
\end{theorem}

\begin{remark}[]\label{rem:schrball}
  Similar results can be proved for other dispersive equations,
  notably for the nonlinear Schr\"{o}dinger equation.
  This will be the topic of further work in preparation.
\end{remark}

The proof of Theorem \ref{the:quasiradial} is based on 
perturbing the radial solution $u(t,x)$
with a small term $w(t,x)$ satisfying the equation
\begin{equation*}
  \square w+|u+w|^{p-1}(u+w)-|u|^{p-1}u=0,
\end{equation*}
which can be written in the form
\begin{equation*}
  \square w+V(t,x)w=w^{2}\cdot F[u,w],
  \qquad
  V(t,x)=p|u|^{p-1}.
\end{equation*}
To solve the last equation, we prove an energy--Strichartz
estimate for the exterior problem with potential
\begin{equation*}
  \square u+V(t,x)u=F,
  \qquad
  u(0,x)=u_{0}(x),
  \qquad
  u_{t}(0,x)=u_{1}(x),
  \qquad
  u(t,\cdot)\vert_{\partial \Omega}=0
\end{equation*}
as an application of the results of \cite{SmithSogge00-a},
\cite{Metcalfe04-a}, \cite{Burq03-a};
note that this part of the argument holds also on arbitrary
non trapping exterior domains. 
The Strichartz estimate for the wave equation with
potential is proved by
reduction to the constant coefficient case, which
is possible since
from the previous results we know that the potential
$V$ has rapid decay in $(t,x)$.
Finally, Theorem \ref{the:wsuniq} is
an adaptation of a weak--strong uniqueness result
due to M.~Struwe \cite{Struwe06}.

The plan of the paper is the following. The linear theory and
the perturbed energy--Strichartz estimates are developed in
Section \ref{sec:line_theo}. 
In Section \ref{sec:globalrad} 
the global radial solution is constructed and
its regularity and decay properties are studied.
In the final Sections \ref{sec:quasiradial},
\ref{sec:weak_stro_uniq}
we prove the main results
Theorems \ref{the:quasiradial} and
Theorem \ref{the:wsuniq}.

\section{Linear theory}\label{sec:line_theo}

Denote by $-\Delta_{D}$ the Dirichlet Laplacian, that is the
nonnegative selfadjoint operator with domain 
$H^{2}(\Omega)\cap H^{1}_{0}(\Omega)$, and by $\Lambda$ its 
nonnegative selfadjoint square root
\begin{equation*}
  \Lambda=(-\Delta_{D})^{1/2},
  \qquad
  D(\Lambda)=H^{1}_{0}(\Omega).
\end{equation*}
We have $\Lambda^{2k}=(-\Delta_{D})^{k}$ for integer $k\ge0$,
and 
\begin{equation*}
  D(\Lambda^{2k})=D(\Delta_{D}^{k})=
  \{f\in
  H^{2k}(\Omega):
  f,\Delta f,\dots ,\Delta^{k-1}f\in H^{1}_{0}(\Omega)
  \},
\end{equation*}
\begin{equation*}
  D(\Lambda^{2k+1})=
  \{f\in
  H^{2k+1}(\Omega):
  f,\Delta f,\dots ,\Delta^{k}f\in H^{1}_{0}(\Omega)
  \}
\end{equation*}
so that
\begin{equation*}
  \textstyle
  D(\Lambda^{k})=
  \{f\in H^{k}(\Omega) \colon \Delta^{j}f\in H^{1}_{0}(\Omega),
  \ 0\le 2j\le k-1\}.
\end{equation*}
We shall make repeated use of the equivalence
\begin{equation*}
  \|\nabla_{x}g\|_{L^{2}(\Omega)}
  =
  \|\Lambda g\|_{L^{2}(\Omega)},
\end{equation*}
valid for $g\in D(\Lambda)=H^{1}_{0}(\Omega)$.
The solution of the special mixed problem
\begin{equation*}
  \square u=0,\qquad u(0,x)=0, \qquad u_{t}(0,x)=u_{1},
  \qquad
  u(t,\cdot)\vert_{\partial \Omega}=0
\end{equation*}
with $u_{1}\in H^{1}_{0}(\Omega)$ can be represented in the form
\begin{equation}\label{eq:defSt}
  S(t)u_{1}:=\Lambda^{-1}\sin(t \Lambda)u_{1}
\end{equation}
Then the solution of the full \emph{linear} mixed problem
on $\mathbb{R}^{+}\times \Omega$
\begin{equation}\label{eq:linear}
  \square u=F(t,x)
  \qquad
  u(0,x)=u_{0}(x),
  \qquad
  u_{t}(0,x)=u_{1}(x),
  \qquad
  u(t,\cdot)\vert_{\partial \Omega}=0
\end{equation}
takes the form
\begin{equation}\label{eq:fullsol}
  \textstyle
  u=S(t)u_{1}+\partial_{t}S(t)u_{0}+
  \int_{0}^{t}S(t-s)F(s)ds,
\end{equation}
where
\begin{equation*}
  \partial_{t}S(t)u_{0}=\cos(t \Lambda)u_{0}.
\end{equation*}

We shall be concerned with several global in time
estimates of $S(t)$ and of the solution to
\eqref{eq:linear}.
We work for positive times $t>0$ only, but it is clear
that all results are time--reversible.
Directly from the spectral theory one gets
\begin{equation}\label{eq:ener1}
  \|\nabla_{x} S(t)g\|_{L^{2}(\Omega)}
  \le\|g\|_{L^{2}(\Omega)},
  \qquad
  \|\partial_{t} S(t)g\|_{L^{2}(\Omega)}\le\|g\|_{L^{2}(\Omega)},
\end{equation}
\begin{equation}\label{eq:ener2}
  \|S(t)g\|_{L^{2}(\Omega)}\le t\|g\|_{L^{2}(\Omega)},
  \qquad
  \|S(t)g\|_{L^{2}(\Omega)}\le
  \|\Lambda^{-1}g\|_{L^{2}(\Omega)}\lesssim
  \|g\|_{L^{\frac{2n}{n+2}}(\Omega)}.
\end{equation}
As a consequence one gets
the basic energy estimate for solutions of \eqref{eq:linear}:
\begin{equation}\label{eq:en0}
  \|\nabla_{t,x}u\|_{L^{\infty}(0,T;L^{2}(\Omega))}
  \lesssim
  \|\nabla_{x}u_{0}\|_{L^{2}}+
  \|u_{1}\|_{L^{2}}+
  \|F\|_{L^{1}(0,T;L^{2}(\Omega))}
\end{equation}
valid for all $T>0$, with a constant independent of $T$.

Higher regularity results require compatibility conditions.
Given the data $(u_{0},u_{1},F)$ we define
recursively the sequence of functions $h_{j}$ as follows:
\begin{equation}\label{eq:lincc}
  h_{0}=u_{0},\qquad
  h_{1}=u_{1},\qquad
  h_{j}=\partial_{t}^{j}u(0,x)=
  \Delta h_{j-2}+\partial_{t}^{j-2}F(0,x),
  \qquad
  j\ge2.
\end{equation}
The function $h_{j}$ is obtained by applying 
$\partial_{t}^{j-2}$ to the equation $u_{tt}=\Delta u+F$.

\begin{definition}[Linear compatibility conditions]\label{def:lcc}
  We say that the data $(u_{0},u_{1},F)$ satisfy the
  \emph{linear compatibility conditions of order $N\ge1$} if
  $(u_{0}, u_{1})\in H^{N+1}(\Omega)\times H^{N}(\Omega)$,
  $F\in C^{k}H^{N-k}(\Omega)$ for $0\le k\le N$,
  and
  \begin{equation*}
    h_{j}\in H^{1}_{0}(\Omega)
    \quad\text{for}\quad 
    0\le j\le N.
  \end{equation*}
\end{definition}

To formulate estimates of $u$ in a compact format
we introduce a few notations.
We write for short for any interval $I \subseteq\mathbb{R}$
and $T\ge0$
\begin{equation*}
  L^{p}_{I}L^{q}=L^{p}(I;L^{q}(\Omega)),
  \qquad
  L^{p}_{T}L^{q}=L^{p}_{[0,T]}L^{q},
  \qquad
  L^{p}L^{q}=L^{p}_{[0,\infty)}L^{q}.
\end{equation*}
Moreover we denote the $L^{p}L^{q}$ norm of all spacetime 
derivatives up to the order $N$ by
\begin{equation}\label{eq:Ypq}
  \textstyle
  \|u\|_{Y_{I}^{p,q;N}}=
  \sum_{j+|\alpha|\le N}
  \|\partial^{j}_{t}\partial^{\alpha}_{x}u\|
  _{L^{p}(I;L^{q}(\Omega))}.
\end{equation}
When $I=[0,T]$ or $I=[0,+\infty)$ we write also
\begin{equation*}
  Y_{T}^{q,r;N}=Y_{[0,T]}^{q,r;N}
  \qquad
  Y^{q,r;N}=Y_{[0,\infty)}^{q,r;N}.
\end{equation*}
The following result is standard, and valid for
general domains $\Omega$ with, say, $C^{1}$ compact boundary.
We use the inequality 
$\|\Lambda^{-1}g\|_{L^{2}}\lesssim\|g\|_{L^{\frac{2n}{n+2}}}$
in the formulation of \eqref{eq:L2en}.

\begin{proposition}[]\label{pro:linear}
  Let $N\ge1$, and assume
  $u_{0}\in H^{N+1}(\Omega)$, $u_{1}\in H^{N}(\Omega)$
  and $F\in C^{k}H^{N-k}(\Omega)$ for $0\le k\le N$ satisfy 
  the linear compatibility conditions of order $N$. 
  Then Problem \eqref{eq:linear} has a unique solution
  belonging to
  \begin{equation*}
    u\in C^{k}H^{N+1-k}\cap C^N H^{1}_{0}
    \quad\text{for all}\quad 0\le k\le N+1.
  \end{equation*}
  The solution satisfies for all $T>0$ the energy estimates
  \begin{equation}\label{eq:L2en}
    \|u\|_{L^{\infty}_{T}L^{2}}
    \lesssim
    \|u_{0}\|_{L^{2}}+
    \|u_{1}\|_{L^{\frac{2n}{n+2}}}+
    \|F\|_{L^{1}_{T}L^{\frac{2n}{n+2}}},
  \end{equation}
  \begin{equation}\label{eq:EnEst}
    \textstyle
    \|\nabla_{t,x} u\|_{Y^{\infty,2;N}_{T}}
    \lesssim
    \|u_{0}\|_{H^{N+1}}
    +
    \|u_{1}\|_{H^{N}}
    +
    \|F\|_{Y^{1,2;N}_{T}}
  \end{equation}
  with a constant independent of $T$.
\end{proposition}

We next recall Strichartz estimates for the exterior
problem, following \cite{SmithSogge00-a},
\cite{Metcalfe04-a}, \cite{Burq03-a}.
These estimates are valid on the exterior of any strictly 
convex obstacle with smooth boundary in $\mathbb{R}^{n}$,
$n\ge2$. With our notations one has
\begin{equation}\label{eq:strich}
  \textstyle
  \|S(t)g\|_{L^{q}L^{r}}\lesssim
  \|g\|_{\dot H^{s}},
  \qquad
  s=\frac 12-\frac 1r+\frac 1q
\end{equation}
provided $n\ge3$ and
\begin{equation}\label{eq:strichad}
  \textstyle
  \frac 2q+\frac{n-1}{r}=\frac{n-1}{2},
  \qquad
  2<q\le \infty,
  \qquad
  2\le r<\frac{2(n-1)}{n-3}.
\end{equation}
A couple $(q,r)$ as in \eqref{eq:strichad} is called
\emph{admissible}; note that the endpoint
$(q,r)=(2,\frac{2(n-1)}{n-3})$ is not included
and it is not known if the estimate holds also in this case.

One can further extend the range of indices by combining 
\eqref{eq:strich} with Sobolev embedding. We
shall focus on the following special case:
\begin{equation}\label{eq:strsob}
  \|S(t)g\|_{L^{q}L^{r}}\lesssim\|g\|_{\dot H^{1}}
\end{equation}
valid for $n\ge3$ and for couples of indices of the form
\begin{equation}\label{eq:strsubadm}
  \textstyle
  q=\frac 2 \delta,
  \qquad
  r=\frac{2n}{n-2-\delta},
  \qquad
  0\le \delta<1.
\end{equation}
Then we have:

\begin{proposition}[]\label{pro:strich}
  The solution $u$ to \eqref{eq:linear} saitisfies,
  for any interval $I$ containing 0, the
  Sobolev--Strichartz estimate
  \begin{equation}\label{eq:sobstr}
    \|u\|_{L^{q}_{I}L^{r}}\lesssim
    \|u_{0}\|_{\dot H^{1}}+
    \|u_{1}\|_{L^{2}}+
    \|F\|_{L^{1}_{I}L^{2}}
  \end{equation}
  provided $n\ge3$ and $(q,r)$ satisfy \eqref{eq:strichad}.
\end{proposition}

\begin{proof}
  The proof is a standard application of the Christ--Kiselev
  Lemma to the representation \eqref{eq:fullsol} of the solution.
\end{proof}

Differentiating \eqref{eq:linear} with respect to $t$
and applying the usual recursive procedure one gets, more
generally, the following higher order estimates;

\begin{proposition}[]\label{pro:strichder}
  Assume the data $(u_{0},u_{1},F)$ of \eqref{eq:linear}
  satisfy the compatibility conditions of order $N\ge1$.
  Then, for all $(q,r)$ as in \eqref{eq:strsubadm}
  and for any interval $I$ of length $|I| \gtrsim1$ 
  containing 0, the solution satisfies the estimates
  \begin{equation}\label{eq:strichder}
    \|u\|_{Y^{q,r;N}_{I}}\lesssim
    \|u_{0}\|_{H^{N+1}}+\|u_{1}\|_{H^{N}}+
    \|F\|_{Y^{1,2;N}_{I}}.
  \end{equation}
\end{proposition}

\begin{proof}
  We give a sketch of the proof.
  Applying $\partial_{t}$ to the equation, by 
  \eqref{eq:sobstr} we get
  \begin{equation*}
    \|u_{t}\|_{L^{q}_{I}L^{r}}\lesssim
    \|u_{0}\|_{\dot H^{2}}+
    \|u_{1}\|_{\dot H^{1}}+
    \|F(0,\cdot)\|_{L^{2}}+
    \|F_{t}\|_{L^{1}_{I}L^{2}}
  \end{equation*}
  since $u_{tt}(0)=\Delta u_{0}+F(0,x)$ from the equation.
  We note that
  $\|F(0,\cdot)\|_{L^{2}}\lesssim\|F\|_{Y^{1,2;1}_{I}}$
  provided the interval has length $|I|\gtrsim1$, and this
  gives the estimate for $u_{t}$. In a similar way one
  can estimate all derivatives $\partial_{t}^{j}u$.
  We next estimate $\Delta u=u_{tt}-F$:
  \begin{equation*}
    \|\Delta u\|_{L^{q}_{I}L^{r}}\le
    \|u_{tt}\|_{L^{q}_{I}L^{r}}+\|F\|_{L^{q}_{I}L^{r}}.
  \end{equation*}
  The $u_{tt}$ term has already been estimated.
  As to the second term, we note that
  \begin{equation*}
    \|F\|_{Y^{1,2;2}_{I}}\gtrsim
      \|F\|_{L^{\infty}_{I}H^{1}}
      \gtrsim
      \|F\|_{L^{\infty}_{I}L^{\frac{2n}{n-2}}}
  \end{equation*}
  which is the endpoint $\delta=0$ in \eqref{eq:strsubadm}.
  Moreover,
  \begin{equation*}
    \|F\|_{Y^{1,2;2}_{I}}\gtrsim
    \|F\|_{L^{1}_{I}H^{2}}+
    \|F\|_{L^{\infty}_{I}H^{1}}
  \end{equation*}
  and by interpolation and Sobolev embedding
  \begin{equation*}
    \|F\|_{Y^{1,2;2}_{I}}\gtrsim
    \|F\|_{L^{2}_{I}H^{3/2}}\gtrsim
    \|F\|_{L^{2}_{I}L^{\frac{2n}{n-3}}}
  \end{equation*}
  which is the other endpoint $\delta=1$ in \eqref{eq:strsubadm}.
  This argument can be modified in the case $n=3$ by using the
  Sobolev embedding into $BMO$ instead of $L^{\infty}$.
  Again by interpolation we get
  \begin{equation*}
    \|F\|_{Y^{1,2;2}_{I}}\gtrsim
    \|F\|_{L^{q}_{I}L^{r}}
  \end{equation*}
  for all $(q,r)$ as in \eqref{eq:strichad}, and we conclude
  \begin{equation*}
    \|\Delta u\|_{L^{q}_{I}L^{r}}
    \lesssim
    \|u_{0}\|_{H^{3}}+\|u_{1}\|_{H^{2}}+\|F\|_{Y^{1,2;2}}.
  \end{equation*}
  Also by interpolation this covers the case $N=1$ of
  \eqref{eq:strichder}.
  For larger values of $N$ one proceeds in a similar way
  by recursion, using the embedding
  $Y^{1,2;N}\hookrightarrow Y^{q,r;m-2}$ just proved for
  all $(q,r)$ as in \eqref{eq:strsubadm}.
\end{proof}

We shall also need estimates for the exterior wave equation
with a time dependent potential. We denote by
$u=S_{V}(t;t_{0})g$ the solution of the mixed problem
with initial data at time $t=t_{0}$
\begin{equation*}
  \square u+V(t,x)u=0,\qquad u(t_{0},x)=0, \qquad 
  u_{t}(t_{0},x)=g,
  \qquad
  u(t,\cdot)\vert_{\partial \Omega}=0
\end{equation*}
and by $u=S_{V}'(t;t_{0})f$ 
(which is not $\partial_{t}S_{V}g$) the solution of
\begin{equation*}
  \square u+V(t,x)u=0,\qquad u(t_{0},x)=f, \qquad 
  u_{t}(t_{0},x)=0,
  \qquad
  u(t,\cdot)\vert_{\partial \Omega}=0.
\end{equation*}
If $V$ is a sufficiently smooth potential
with good behaviour at infinity,
the existence and uniqueness of a solution is standard.
The solution of the full problem
\begin{equation}\label{eq:fullV}
  \square u+V(t,x)u=F(t,x),\qquad u(t_{0},x)=u_{0}, \qquad 
  u_{t}(t_{0},x)=u_{1},
  \qquad
  u(t,\cdot)\vert_{\partial \Omega}=0.
\end{equation}
can be represented by Duhamel as
\begin{equation}\label{eq:fullduh}
  \textstyle
  u=S_{V}'(t;t_{0})u_{0}+S_{V}(t;t_{0})u_{1}
  +\int_{0}^{t}S_{V}(t;s)F(s)ds.
\end{equation}
Note also that it is not necessary to modify the compatibility
conditions, provided the potential $V$ has a sufficient 
regularity. Indeed the correct condition would require
\begin{equation*}
  \textstyle
  h_{j}=\partial_{t}^{j}u(0,x)=
  \Delta h_{j-2}+
  \sum\limits_{\ell=0}^{j-2}
  \partial_{t}^{j-2-\ell}V(0,x)h_{\ell}
  +\partial_{t}^{j-2}F(0,x)
  \in H^{1}_{0}
\end{equation*}
but the term $\partial_{t}^{j-2-\ell}V(0,x)h_{\ell}$ is
already in $H^{1}_{0}$ since $h_{\ell}\in H^{1}_{0}$,
and can be omitted.
By Duhamel we can write $S_{V}(t;s)$, $S_{V}'(t,s)$ as 
perturbations of $S(t)$, $\partial_{t}S(t)$:
\begin{equation}\label{eq:SVS}
  \textstyle
  S_{V}(t;t_{0})=
  S(t-t_{0})-
  \int_{t_{0}}^{t}S(t-s)V(s,x)S_{V}(s;t_{0})ds,
\end{equation}
\begin{equation}\label{eq:SpVS}
  \textstyle
  S_{V}'(t;t_{0})=
  \partial_{t}
  S(t-t_{0})-
  \int_{t_{0}}^{t}S(t-s)V(s,x)S_{V}(s;t_{0})ds.
\end{equation}

\begin{proposition}[Perturbed energy--Strichartz estimate]
  \label{pro:enerstr}
  Let $n\ge3$, $m\ge1$. Assume the data $(u_{0},u_{1},F)$ satisfy
  the compatibility conditions of order $m$, and that
  \begin{equation}\label{eq:assVen}
    \|V\|_{Y^{1,n;m}}
    <\infty.
  \end{equation}
  Then for any interval $I$ containing $t_{0}$
  the solution of Problem \eqref{eq:fullV} satisfies
  \begin{equation}\label{eq:L2enV}
    \|u\|_{L^{\infty}_{I}L^{2}}
    \lesssim
    \|u_{0}\|_{L^{2}}+
    \|u_{1}\|_{L^{\frac{2n}{n+2}}}+
    \|F\|_{L^{1}_{I}L^{\frac{2n}{n+2}}},
  \end{equation}
  \begin{equation}\label{eq:EnEstV}
    \textstyle
    \|\nabla_{t,x} u\|_{Y^{\infty,2;m}_{I}}
    +
    \|u\|_{Y^{q,r;m}_{I}}
    \lesssim
    \|u_{0}\|_{H^{m+1}}
    +
    \|u_{1}\|_{H^{m}}
    +
    \|F\|_{Y^{1,2;m}_{I}}
  \end{equation}
  provided the couple $(q,r)$ is of the form \eqref{eq:strsubadm}.
\end{proposition}

\begin{proof}
  We can assume $I=I(t)=[t_{0},t]$; the proof for $t<t_{0}$
  is identical.
  Since $u$ solves $\square u=F-Vu$, by \eqref{eq:L2en}
  we get
  \begin{equation*}
    \textstyle
    \|u(t)\|_{L^{2}}
    \lesssim
    \|u_{0}\|_{L^{2}}+
    \|u_{1}\|_{L^{\frac{2n}{n+2}}}
    +
    \|F\|_{L^{1}_{I}L^{\frac{2n}{n+2}}}
    +\int_{t_{0}}^{t}\|Vu\|_{L^{\frac{2n}{n+2}}}ds.
  \end{equation*}
  Noting that
  \begin{equation*}
    \textstyle
    \int_{t_{0}}^{t}\|Vu\|_{L^{\frac{2n}{n+2}}}ds
    \le
    \int_{t_{0}}^{t}\|V(s)\|_{L^{n}}\|u(s)\|_{L^{2}}ds
  \end{equation*}
  and $a(s)=\|V(s)\|_{L^{n}}$ is integrable by \eqref{eq:assVen},
  by Gronwall's Lemma we deduce \eqref{eq:L2enV}.

  In a similar way, by \eqref{eq:EnEst} and \eqref{eq:strichder}
  we can write
  \begin{equation*}
    \textstyle
    \|\nabla_{t,x} u\|_{Y^{\infty,2;m}_{I(t)}}+
    \|u\|_{Y^{q,r;m}_{I(t)}}
    \lesssim
    \|u_{0}\|_{H^{m+1}}
    +
    \|u_{1}\|_{H^{m}}
    +
    \|F\|_{Y^{1,2;m}_{I(t)}}+
    \|Vu\|_{Y^{1,2;m}_{I(t)}}.
  \end{equation*}
  We have 
  \begin{equation*}
    \textstyle
    \|Vu\|_{Y^{1,2;m}_{I(t)}}
    \lesssim
    \int_{t_{0}}^{t}
    b(s)\|u\|_{Y^{\infty,\frac{2n}{n-2};k}_{I(s)}}ds
    \lesssim
    \int_{t_{0}}^{t}
    b(s)\|\nabla _{x}u\|_{Y^{\infty,2;k}_{I(s)}}ds
  \end{equation*}
  where
  \begin{equation*}
    \textstyle
    b(s)=
    \sum_{j+|\alpha|\le k}
    \|\partial^{j}_{t}\partial^{\alpha}_{x}V(s)\|_{L^{n}}
  \end{equation*}
  is integrable on $\mathbb{R}$ by \eqref{eq:assVen}.
  Using again Gronwall's inequality we obtain \eqref{eq:EnEstV}.
\end{proof}

\section{The global radial solution}\label{sec:globalrad}

This section is devoted to the proof of
Proposition \ref{pro:radial} and
Theorems \ref{the:decayrad}, \ref{the:regul}.
We begin with a few preliminary results on the mixed problem
\begin{equation}\label{eq:gensemil}
  \square u+f(u)=0,
  \qquad u(0,x)=u_{0}, \qquad u_{t}(0,x)=u_{1},\qquad
  u(t,\cdot)\vert_{\partial \Omega}=0.
\end{equation}
For data of low regularity, a solution to \eqref{eq:gensemil}
is intended to be a solution of the integral equation
\begin{equation}\label{eq:genint}
  \textstyle
  u=S(t)u_{1}+\partial_{t}S(t)u_{0}-
  \int_{0}^{t} S(t-s)f(u(s))ds
\end{equation}
with $S(t)$ as in \eqref{eq:defSt}. We will give only sketchy
proofs of standard results, which are virtually identical
to the corresponding ones for semilinear wave equations on 
$\mathbb{R}^{n}$
(for which we refer e.g. to Chapter 6 of \cite{ShatahStruwe98-a}).

\begin{lemma}[]\label{lem:lipschf}
  Assume $f:\mathbb{R}\to \mathbb{R}$ is (globally)
  Lipschitz with $f(0)=0$.
  Then for any initial data
  $(u_{0}, u_{1})\in H^{1}_{0}\times L^{2}(\Omega) $
  Problem \eqref{eq:gensemil} has a unique solution
  $u(t,x)$ in $C H^{1}_{0}(\Omega)\cap C^{1}L^{2}(\Omega)$.
  The solution satisfies the energy bound for all $t>0$
  \begin{equation}\label{eq:enerfu}
    \|u(t,\cdot)\|_{H^{1}(\Omega)}
    +
    \|\partial _{t}u(t,\cdot)\|_{L^{2}(\Omega)}
    \le
    C(f,t)
    \Bigl[
    \|u_{0}\|_{H^{1}(\Omega)}
    +
    \|u_{1}\|_{L^{2}(\Omega)}
    \Bigr].
  \end{equation}
\end{lemma}

\begin{proof}
  Apply a contraction argument in the space
  $C([0,T]; H^{1}_{0}(\Omega))
    \cap C^{1}([0,T];L^{2}(\Omega))$
  to \eqref{eq:genint},
  with $T>0$ sufficiently small, using the energy estimates
  \eqref{eq:ener1}, \eqref{eq:ener2}. The lifespan $T$ depends
  only on the $H^{1}\times L^{2}$ norm of the data, thus we can
  iterate to a global solution. Estimate \eqref{eq:enerfu}
  is a byproduct of the proof.
\end{proof}

\begin{lemma}[]\label{lem:H2reg}
  Consider the solution $u$ constructed in Lemma \ref{lem:lipschf}. 
  Assume in addition that $f\in C^{2}$,
  $u_{0}\in H^{2}(\Omega)\cap H^{1}_{0}(\Omega)$,
  $u_{1}\in H^{1}_{0}(\Omega)$. Then $u$ belongs to
  $C^{2}L^{2}(\Omega)\cap C^{1}H^{1}_{0}(\Omega)\cap
    C H^{2}(\Omega)$ and solves \eqref{eq:gensemil} in both
  distributional and a.e. sense.

  If we further assume that $0\le f(s)s\lesssim F(s)$ for
  $s\in \mathbb{R}$, where $F(s)=\int_{0}^{s}f(\sigma)d\sigma$,
  then the solution satisfies for all times the energy identity
  \begin{equation}\label{eq:globenid}
    \textstyle
    E(t)=E(0),
    \qquad
    E(t):=\int_{\Omega}
    \left[
      \frac 12|\nabla_{x}u|^{2}+\frac 12|u_{t}|^{2}
      +F(u)
    \right]
    dx.
  \end{equation}
\end{lemma}

\begin{proof}
  Differentiate the equation once w.r.to space variables,
  noting that $S(t)$ commutes with spatial derivatives. 
  The nonlinear term produces a term $f'(u)\partial u$ where 
  $f'(u)$ is uniformly bounded; then the (linear) energy 
  estimate gives
  $D^{2}_{x}u\in C L^{2}(\Omega)$. The estimate for $u_{tt}$ is
  deduced from the equation itself.
\end{proof}

Note that the assumption $0\le sf(s)\lesssim F(s)$ is sufficient
to prove the existence of a global weak solution for data
in $H^{1}_{0}(\Omega)\times L^{2}(\Omega)$, even if $f$ is
not Lipschitz (Segal's Theorem). 
This can be proved like in the case of the whole space 
$\mathbb{R}^{n}$ by
approximating $f$ with a sequence of truncated Lipschitz functions
and using weak compactness.
The weak solution thus constructed satisfies then
a weaker energy inequality $E(t)\le E(0)$ (proved using
Fatou's Lemma). We shall not need this variant in the sequel.

For smoother data, one can prove 
a local existence theorem which does not
require a global Lipschitz condition, similarly to the case 
$\Omega=\mathbb{R}^{n}$. However, one must assume suitable
compatibility conditions, analogous to the linear ones from
Definition \ref{def:lcc}. 
Define formally a sequence of functions
$\psi_{j}$, $j\ge0$ as follows: differentiating the equation
$u_{tt}=\Delta u+f(u)$ with respect to time, set
\begin{equation*}
  \psi_{0}=u(0,x)=u_{0},
  \quad
  \psi_{1}=u_{t}(0,x)=u_{1},
  \quad
  \psi_{j}=\partial_{t}^{j}u(0,x)=
  \Delta \psi_{j-2}+\partial_{t}^{j-2}(f(u))\vert_{t=0},
\end{equation*}
where values of $\partial_{t}^{k}u(0,x)$
for $0\le k\le j-2$, required to compute $\psi_{j}$,
are set recursively equal to $\psi_{k}$. 
For instance,
\begin{equation*}
  \psi_{2}=\Delta u(0,x)+ f(u(0,x))=
  \Delta \psi_{0}+ f(\psi_{0}),
\end{equation*}
\begin{equation*}
  \psi_{3}=\Delta \psi_{1}+f(u)u_{t}=
  \Delta \psi_{1}+
  f'(\psi_{0})\psi_{1},
\end{equation*}
\begin{equation*}
  \psi_{4}=\Delta \psi_{2}+f''(u)u_{t}^{2}+f'(u)u_{tt}=
  \Delta \psi_{2}+
  f''(\psi_{0})\psi_{1}^{2}+f'(\psi_{0})\psi_{2}
\end{equation*}
and so on. Then we have:

\begin{definition}[Nonlinear compatibility conditions]
\label{def:nlcc}
  We say that the data $(u_{0},u_{1},f)$ satisfy
  the \emph{nonlinear compatibility conditions of order $N\ge1$}
  if $~(u_{0},u_{1})\in H^{N+1}(\Omega)\times H^{N}(\Omega)$,
  $f\in C^{N-2}(\mathbb{R};\mathbb{R})$ and we have
  \begin{equation}\label{eq:nlcc}
    \psi_{j}\in H^{1}_{0}(\Omega)
    \quad\text{for}\quad 
    0\le j\le N.
  \end{equation}
\end{definition}

\begin{remark}[]\label{rem:simplecc}
  Conditions \eqref{eq:nlcc} are implied by a number of simpler
  assumptions on the data. For instance, if $f\in C^{N-2}$
  and one assumes
  \begin{equation}\label{eq:strong}
    u_{0}\in H^{\lfloor \frac n2 \rfloor+N+1}_{0},
    \quad
    u_{1}\in H^{\lfloor \frac n2 \rfloor+N}_{0}
  \end{equation}
  then one checks easily that
  $(u_{0},u_{1},f)$ satisfiy the nonlinear compatibility conditions
  of order $N$. 
\end{remark}

\begin{lemma}[Local existence]\label{lem:localfu}
  Let $N>  \frac n2 $,
  $(u_{0},u_{1})\in H^{N+1}(\Omega)\times H^{N}(\Omega)$,
  $f\in C^{N}$ and assume $(u_{0},u_{1},f)$ satisfy the 
  nonlinear compatibility conditions
  of order $N$. Then there exists $0<T\le +\infty$, 
  depending only on the
  $H^{N+1}\times H^{N}$ norm of $(u_{0},u_{1})$, such that Problem
  \eqref{eq:gensemil} has a unique solution
  $u\in C^{k}([0,T);H^{N+1-k}(\Omega))$, $0\le k\le N+1$.
  The solution belongs to $C^{N}([0,T);H^{1}_{0}(\Omega))$.

  Moreover, if $T^{*}$ is the maximal time of existence of such
  a solution, then either $T^{*}=+\infty$ or 
  $\|u(t,\cdot)\|_{L^{\infty}}\to \infty$ as $t \uparrow T^{*}.$
\end{lemma}

\begin{proof}
  The existence part is completely standard; 
  it is usually proved for
  more general quasilinear equations, which require higher
  smoothness of the data; see e.g. 
  Theorem 3.5 in \cite{ShibataTsutsumi86}, where local existence
  is proved for a nonlinear term of the form
  $f(t,x,\partial_{t}^{j}\partial_{x}^{\alpha}u)$ 
  with $j+|\alpha|\le2$, $j\le1$ (and a regularity of order
  $\lfloor \frac n2 \rfloor+8$ is imposed on the data).
  The proof is based on a contraction mapping argument,
  combined with ~Moser type estimates of the nonlinear term.
  The final blow up alternative in the statement is
  a byproduct of the proof.
\end{proof}

\subsection{Proof of Proposition \ref{pro:radial}}
Fix $M>0$ and define
$f_{M}(s)=\min\{|s|,M\}^{p-1}s$. Then $f_{M}$ is Lipschitz and the
problem
\begin{equation}\label{eq:fMu}
  \square u+f_{M}(u)=0,
  \qquad u(0,x)=u_{0}, \qquad u_{t}(0,x)=u_{1},\qquad
  u(t,\cdot)\vert_{\partial \Omega}=0
\end{equation}
has a global, unique, radially symmetric solution
$u\in C^{2}L^{2}(\Omega)\cap C^{1}H^{1}_{0}(\Omega)\cap
  C H^{2}(\Omega)$ satisfying the bound \eqref{eq:globenid}
with $F=F_{M}=\int_{0}^{s}f_{M}$. Combining \eqref{eq:globenid}
with \eqref{eq:radialstrauss} we get
\begin{equation}\label{eq:apri}
  |x|^{\frac n2-1}|u(t,x)|
  \le C_{0}K,
  \qquad
  K:=
  \|u_{0}\|_{\dot H^{1}}+\|u_{1}\|_{L^{2}}+
  \|u_{0}\|_{L^{p+1}}^{\frac{p+1}{2}}
\end{equation}
for some universal constant $C_{0}$. Since $|x|\ge1$ on $\Omega$,
this gives
\begin{equation*}
  |u(t,x)|\le C_{0}K
\end{equation*}
and if we choose $M=C_{0}K+1$ we see that 
$f_{M}(u(t,x))=f(u(t,x))$, i.e., $u(t,x)$ is a 
global solution of the
untruncated problem
\begin{equation}\label{eq:untrunc}
  \square u+|u|^{p-1}u=0,
  \qquad u(0,x)=u_{0}, \qquad u_{t}(0,x)=u_{1},\qquad
  u(t,\cdot)\vert_{\partial \Omega}=0.
\end{equation}
The same argument guarantees also uniqueness of radially symmetric
solutions. More generally, a 
solution in $H^{2}\cap L^{\infty}_{loc}$ with the same initial data
must coincide
with the radial one, as it follows by a localization
argument and finite speed of propagation.
The proof of Proposition \ref{pro:radial} is concluded.

\subsection{Proof of Theorem \ref{the:decayrad}}
We now prove the decay estimate \eqref{eq:decaysol},
using the Penrose transform.
We recall its definition.
Describe the sphere $\mathbb{S}^{n}$ using
coordinates $(\alpha,\theta)$ with $\alpha\in (0,\pi)$
and $\theta\in \mathbb{S}^{n-1}$, as
$\mathbb{S}^{n}_{\alpha,\theta}
  =(0,\pi)_{\alpha}\times \mathbb{S}^{n-1}_{\theta}$.
Denoting with $d \theta_{\mathbb{S}^{n-1}}^{2}$ the metric
of $\mathbb{S}^{n-1}_{\theta}$,
the metric on $\mathbb{S}^{n}$ can then be written as
\begin{equation*}
  d \alpha^{2}+(\sin \alpha)^{2}d \theta_{\mathbb{S}^{n-1}}^{2}.
\end{equation*}
Similarly, on $\mathbb{R}^{n}_{x}=\mathbb{R}^{+}_{r}
  \times \mathbb{S}^{n-1}_{\theta}$, use polar coordinates
$(r,\theta)$ with $r\in(0,+\infty)$ and
$\theta\in \mathbb{S}^{n-1}$, so that the euclidean metric
can be written
$dr^{2}+r^{2}d \theta_{\mathbb{S}^{n-1}}$.
Then we can define the Penrose map
$\Pi:
\mathbb{R}_{t}\times \mathbb{R}_{x}\to
\mathbb{R}_{T}\times \mathbb{S}^{n}$ 
as
\begin{equation*}
  \Pi:(t,r,\theta)\mapsto(T,\alpha,\theta)
\end{equation*}
where
  \begin{equation*}
    T=\arctan(t+r)+\arctan(t-r),
    \qquad
    \alpha= \arctan(t+r)-\arctan(t-r).
\end{equation*}
The map $(t,r)\mapsto(T,\alpha)$ takes the quadrant 
\begin{equation*}
  \{(t,r):t\ge0,\ r\ge0\}
\end{equation*}
to the triangle
\begin{equation*}
  \{(T,\alpha):T\ge0,\ 0\le \alpha<\pi-T\}
\end{equation*}
so that $\Pi$ maps $\mathbb{R}^{+}\times \mathbb{R}^{n}$
to the positive half of the Einstein diamond
\begin{equation*}
  \mathbb{E}^{+}=
  \{(T,\alpha,\theta):
  T\ge0,\ 0\le \alpha<\pi-T,\ \theta\in \mathbb{S}^{n-1}
  \}
\end{equation*}
The boundary
$|x|=1$ i.e. $r=1$ is mapped by $\Pi$ to a curve described
parametrically by
\begin{equation}\label{eq:gacur}
  (T,\alpha)=
  (\arctan(t+1)+\arctan(t-1),
  \arctan(t+1)-\arctan(t-1))
\end{equation}
for $t\ge0$.
We denote this curve by 
\begin{equation*}
  \alpha=\Gamma(T),
  \qquad
  T\in[0,\pi)
  \qquad\text{or}\qquad 
  T=\gamma(\alpha),\qquad \alpha\in[0,\pi)
\end{equation*}
(i.e., $\gamma=\Gamma^{-1}$). One has the explicit 
(but not particularly useful) formulas
\begin{equation*}
  \textstyle
  \Gamma(T)=\frac \pi4+\arcsin(2^{-\frac 12}\cos T),
  \qquad
  \gamma(\alpha)=\arccos(\sin \alpha-\cos \alpha).
\end{equation*}
It is easy to check that
\begin{equation}\label{eq:gam1}
  \textstyle
  \gamma'(\alpha)<-1.
\end{equation}
We denote by $\omega$ the conformal factor
\begin{equation*}
  \omega=\cos T+\cos \alpha=
  \frac{2}{\langle t+r \rangle \langle t-r\rangle },
  \qquad
  \langle s\rangle =(1+s^{2})^{1/2}.
\end{equation*}
Note that $\omega>0$ on $\mathbb{E}^{+}$.
The inverse of $\Pi$ (defined on $\mathbb{E}^{+}$) can be written
\begin{equation*}
  (t,r,\theta)=\Pi^{-1}(T,\alpha,\theta)=
  (\omega^{-1}\sin T,\omega^{-1}\sin \alpha,\theta).
\end{equation*}
We define a new function $U(T,\alpha,\theta)$ via
\begin{equation}\label{eq:transfP}
  u(t,r \theta)=
  \omega^{\frac{n-1}{2}}
  U\circ \Pi(t,r,\theta).
\end{equation}
Since $u,U$ are independent of $\theta$,
we shall write simply $U(T,\alpha)$. 
Commuting $\square$ with $\Pi$ gives
\begin{equation*}
  \textstyle
  \square u = 
  \square_{\mathbb{R}_{t}\times \mathbb{R}^{n}}
    (\omega^{\frac{n-1}{2}} U \circ \Pi )= 
  \omega^{\frac{n+3}{2}}
  (\square_{\mathbb{R}_{T}\times 
     \mathbb{S}^{n}_{\alpha,\theta}}U+
    \frac{(n-1)^{2}}{4}U)
  \circ\Pi
\end{equation*}
so that $U$ is a solution of the equation
\begin{equation}\label{eq:WEUpen}
  \textstyle
  \square_{\mathbb{R}\times \mathbb{S}^{n}} U
  +\frac{(n-1)^{2}}{4}U
  +\omega^{\nu}|U|^{p-1}U=0,
  \qquad \nu=\frac{n-1}{2}p-\frac{n+3}{2}
\end{equation}
on the subset $\mathbb{E}_{\Omega}$ 
of $\mathbb{R}\times \mathbb{S}^{n}$
given by the conditions 
\begin{equation*}
  \mathbb{E}_{\Omega}
  :=
  \{(T,a):
  0\le T<\pi,
  \ 
  \Gamma(T)\le \alpha<\pi-T
  \}
\end{equation*}
which is the image of $\mathbb{R}^{+}_{t}\times \Omega$
via $\Pi$.
We introduce also the notation
\begin{equation*}
  D_{T}=\{(\alpha,\theta)\in \mathbb{S}^{n}:
  \Gamma(T)\le \alpha<\pi-T\},
  \qquad
  0\le T<\pi
\end{equation*}
for the slice of $\Pi(\mathbb{R}^{+}\times \Omega)$ at time $T$.
Note that in coordinates, equation \eqref{eq:WEUpen} reads
\begin{equation*}
  \textstyle
  \partial_{T}^{2}U
  -
  \partial_{\alpha}^{2}U
  -
  (n-1)
  \frac{\cos \alpha}{\sin \alpha}\partial_{\alpha}U
  +
  \frac{(n-1)^{2}}{4}U
  +
  \omega^{\nu}|U|^{p-1}U=0.
\end{equation*}
We plan to extend the solution beyond the line
$T+\alpha=\pi$, i.e., in the region where $\omega<0$.
Thus we consider the following extended equation on 
$(T,\alpha)\in[0,\pi]^{2}$:
\begin{equation}\label{eq:exteq}
  \textstyle
  \partial_{T}^{2}U
  -
  \partial_{\alpha}^{2}U
  -
  (n-1)
  \frac{\cos \alpha}{\sin \alpha}\partial_{\alpha}U
  +
  \frac{(n-1)^{2}}{4}U
  +
  \widetilde{\omega}^{\nu}|U|^{p-1}U=0
\end{equation}
where we have replaced $\omega$ by
\begin{equation*}
  \widetilde{\omega}:=
  \begin{cases}
    \omega &\text{if $ T+\alpha\le\pi $,}\\
    0 &\text{if $ T+\alpha>\pi $.}
  \end{cases}
\end{equation*}
The solution $U$ satisfies the identity
\begin{equation}\label{eq:iden2}
\begin{split}
  \partial_{T}
  &
  \textstyle
  \left\{
  (\sin \alpha)^{n-1}
  \left(
  \frac{|\partial_{T}U|^{2}+|\partial_{\alpha} U|^{2}}2
  +
  \frac{\widetilde{\omega}^{\nu}}{p+1}|U|^{p+1}
  +
  \frac{(n-1)^{2}}{8}|U|^{2}
  \right)
  \right\}=
  \\
  &
  \qquad \qquad
  \qquad \qquad
  \textstyle
  =
  \partial_{\alpha}
  \left\{
    (\sin \alpha)^{n-1}
    \partial_{\alpha}U \partial_{T}U
  \right\}
  -\frac{\nu \widetilde{\omega}^{\nu-1}}{p+1} 
  (\sin \alpha)^{n-1}
  (\sin T) |U|^{p+1}.
  \end{split}
\end{equation}

We can now extend $U$ to a larger domain in the cylinder
$\mathbb{R}_{T}\times \mathbb{S}^{n}$. Recall that the data
$u_{0},u_{1}$ satisfy $C_{M}(u_{0},u_{1})<\infty$ with 
$C_{M}(u_{0},u_{1})$ as
in \eqref{eq:weighsob}. Thus if we fix a smooth cutoff function
$\chi(x)$ equal to 0 for $|x|\le M+1$ and equal to $1$\
for $|x|\ge M+2$, we have
\begin{equation*}
  \|\chi u_{0}\|_{H^{N_{0}+1,N_{0}}}
  +
  \|\chi u_{1}\|_{H^{N_{0},N_{0}+1}}<\infty.
\end{equation*}
Denote by $\widetilde{U}_{0},\widetilde{U}_{1}$ the
functions obtained by applying the transformation
\eqref{eq:transfP} to $\chi u_{0}, \chi u_{1}$ respectively
(with $t=0$).
By the first Lemma in Section 4 of \cite{Christodoulou86-b}
we have then
\begin{equation*}
  \|\widetilde{U}_{0}\|_{H^{N_{0}+1}(\mathbb{S}^{n})}
  +
  \|\widetilde{U}_{1}\|_{H^{N_{0}}(\mathbb{S}^{n})}
  <\infty.
\end{equation*}
In order to solve \eqref{eq:WEUpen} locally via the energy method
we require that the coefficient $\widetilde{\omega}^{\nu}$ 
be sufficiently
smooth i.e. $\widetilde{\omega}^{\nu}\in C^{N_{0}}$. This is true 
as soon as
\begin{equation*}
  \textstyle
  \nu=\frac{n-1}{2}p-\frac{n+3}{2}>N_{0}
  \quad\text{which is implied by}\quad p>\frac{11}{2}.
\end{equation*}
Then a standard local existence result guarantees the existence
of a local solution $\widetilde{U}$ to equation \eqref{eq:WEUpen}
with data $\widetilde{U}_{0},\widetilde{U}_{1}$
on some strip 
$[0,\delta)\times \mathbb{S}^{n}$. 
The lifespan $\delta$, which can be assumed $\ll1$, depends
only on
\begin{equation}\label{eq:depd}
  \delta=\delta(C_{M}(u_{0},u_{1}),n,p)
\end{equation}
where $C_{M}(u_{0},u_{1})$ was defined in \eqref{eq:weighsob}.
Comparing $\widetilde{U}$
with the solution $U$ constructed above, and noting that equation
\eqref{eq:WEUpen} has finite speed of propagation equal to 1,
by local uniqueness we see that $U, \widetilde{U}$ must coincide
on the forward dependence domain emanating from
the set 
\begin{equation*}
  T=0,\qquad 2 \arctan(M+2)<\alpha<\pi.
\end{equation*}
Thus we can glue the two solutions, at least for
$0<~T<\delta/2$, and we obtain an extended solution
of \eqref{eq:WEUpen}, which we denote again $U(T,\alpha)$,
defined on the larger domain
\begin{equation*}
  \textstyle
  \mathbb{E}^{+}\cup ([0,\frac \delta2)\times \mathbb{S}^{n}).
\end{equation*}
We next prove that the energy of $U$, defined as
\begin{equation*}
  \textstyle
  E(T)=\int_{\Gamma(T)}^{\pi}
  (\sin \alpha)^{n-1}
  \left(
  \frac{|\partial_{T}U|^{2}+|\partial_{\alpha}U|^{2}}2
  +
  \frac{\widetilde{\omega}^{\nu}}{p+1}|U|^{p+1}
  +
  \frac{(n-1)^{2}}{8}|U|^{2}
  \right)
  d \alpha
\end{equation*}
remains bounded for $0\le T<\delta/2$.
To this end we integrate identity \eqref{eq:iden2} on a slice
\begin{equation*}
  T_{1}<T<T_{2},
  \quad
  \Gamma(T)<\alpha<\pi,
\end{equation*}
for arbitrary times $0\le T_{1}\le T_{2}<\delta/2$.
Dropping negative terms at the RHS, 
we are left with the inequality
\begin{equation}\label{eq:interm}
  \textstyle
  E(T_{2})-E(T_{1})
  \le
  (\sin \alpha)^{n-1}
  \Bigl[
  \nu_{\alpha}
  \partial_{\alpha}U \partial_{T}U
  -
  \frac12\nu_{T}\left(|\partial_{T}U|^{2}
  +|\partial_{\alpha}U|^{2}\right)
  \Bigr]_{T=\gamma(\alpha)}
\end{equation}
where $\nu_{\alpha},\nu_{T}$ are the components of the exterior
normal to the curve $T=\gamma(\alpha)$, i.e.,
 \begin{equation*}
   \nu_{T}=-\frac{1}{\sqrt{1+\gamma'(\alpha)^{2}}},
   \qquad
   \nu_{\alpha}
   =\frac{\gamma'(\alpha)}{\sqrt{1+\gamma'(\alpha)^{2}}}.
 \end{equation*} 
The Dirichlet condition $U(\gamma(\alpha),\alpha)=0$ 
along the curve implies
\begin{equation*}
  (\partial_{\alpha}U+\gamma'(\alpha)\partial_{T}U)
    \vert_{T=\gamma(\alpha)}=0.
\end{equation*}
Thus the RHS of \eqref{eq:interm} is equal to
\begin{equation*}
  =(\sin \alpha)^{n-1}
  \frac{1-\gamma'(\alpha)^{2}}{2\sqrt{1+\gamma'(\alpha)^{2}}}
  |\partial_{T}U|^{2}.
\end{equation*}
Recalling \eqref{eq:gam1}, we see that the
RHS of \eqref{eq:interm} is negative, and 
we conclude that the energy is nonincreasing as claimed:
\begin{equation}\label{eq:enest}
  E(T)\le E(0)
  \quad\text{for all}\quad 
  0\le T<\delta/2.
\end{equation}

Now, consider the set, which is a dependence domain
for \eqref{eq:WEUpen}
(keep in mind that
the speed of propagation for \eqref{eq:WEUpen} is exactly 1):
\begin{equation*}
  \textstyle
  \mathbb{D}=
  \{(T,\alpha):
    \frac \delta3\le T<\pi,\ 
    \Gamma(T)<\alpha<\pi-T+\frac \delta4
  \}.
\end{equation*}
We have already extended $U$ to the part of $\mathbb{D}$
in the time strip $\delta/3<T<\delta/2$, and our next goal
is to prove that $U$ can be extended to a bounded solution
of \eqref{eq:WEUpen} on the whole set $\mathbb{D}$.
Clearly, it is sufficient to prove an a priori $L^{\infty}$
bound of the solution on this domain in order to achieve
the result via a continuation argument.

To this end we prove an energy estimate similar to the
previous one, but now we integrate identity \eqref{eq:iden2} 
over the slice
\begin{equation*}
  T_{1}<T<T_{2},
  \quad
  \Gamma(T)<\alpha<\pi-T+\delta/4,
\end{equation*}
where
$  \delta/3\le T_{1}<T_{2}<\pi$ are fixed, and we denote
by $F(T)$ the energy
\begin{equation*}
  \textstyle
  F(T):=
  \int_{\Gamma(T)}^{\pi-T+\delta/4}
  (\sin \alpha)^{n-1}
  \left(
  \frac{|\partial_{T}U|^{2}+|\partial_{\alpha}U|^{2}}2
  +
  \frac{\widetilde{\omega}^{\nu}}{p+1}|U|^{p+1}
  +
  \frac{(n-1)^{2}}{8}|U|^{2}
  \right)d \alpha.
\end{equation*}
After integration of \eqref{eq:iden2},
the terms on the side $\alpha=\pi-T+\delta/4$
give a negative contribution at the RHS which can be
dropped, as in the standard energy estimate,
since the speed of propagation is 1.
Proceeding as before, we are left with the inequality
\begin{equation*}
  \textstyle
  F(T_{2})-F(T_{1})
  \le
  (\sin \alpha)^{n-1}
  \Bigl[
  \nu_{\alpha}
  \partial_{\alpha}U \partial_{T}U
  -
  \frac12\nu_{T}\left(|\partial_{T}U|^{2}
  +|\partial_{\alpha}U|^{2}\right)
  \Bigr]_{T=\gamma(\alpha)}
\end{equation*}
and again the RHS here is negative thanks to  \eqref{eq:gam1}.
We conclude that the energy $F(T)$ is nonincreasing:
\begin{equation*}
  F(T)\le F(\delta/3)
  \quad\text{for all}\quad 
  \delta/3\le T<\pi.
\end{equation*}
Since $F(\delta/3)\le E(\delta/3)$, by \eqref{eq:enest} we conclude
\begin{equation}\label{eq:enest2}
  F(T)\le E(0)
  \quad\text{for all}\quad 
  \delta/3\le T<\pi.
\end{equation}
To proceed, we need a Lemma:

\begin{lemma}[]\label{lem:hardySn}
  Let $I \subseteq (0,+\infty)$ be a bounded interval and
  $n\ge3$. Then for any $V\in H^{1}(I)$ we have
  \begin{equation}\label{eq:inqI}
    \sup_{s\in I}s^{\frac n2-1}|V(s)|
    \lesssim
    \textstyle
    (\int_{I}s^{n-1}|V'(s)|^{2}ds)^{\frac12}
    +
    |I|^{-1}
    (\int_{I}s^{n-1}|V(s)|^{2}ds)^{\frac12}
  \end{equation}
  with a constant independent of $I$ and $V$.
\end{lemma}

\begin{proof}[Proof of the Lemma]
  Pick any points $\alpha,\beta\in I$ with $\beta\ge \alpha/3$.
  We first prove that
  \begin{equation}\label{eq:alphbe}
    \textstyle
    \alpha^{\frac n2-1}|V(\alpha)|
    \lesssim 
    \left(
      \int_{I}s^{n-1}|V'|^{2}
    \right)^{\frac 12}
    + \beta^{\frac n2-1}|V(\beta)|.
  \end{equation}
  We have two cases: either $\alpha\le \beta$ or 
  $\alpha\ge \beta\ge \alpha/3$.
  If $\alpha\le \beta$, we write
  \begin{equation*}
    \textstyle
    |V(\alpha)|\le
    \int_{\alpha}^{\beta}|V'|+|V(\beta)|
    \le
    \left(
      \int_{\alpha}^{\beta}s^{n-1}|V'|^{2}
    \right)^{\frac 12}
    (\int_{\alpha}^{\beta}s^{1-n})^{\frac 12}
    +|V(\beta)|.
  \end{equation*}
  Using the inequality
  $\int_{\alpha}^{\beta}s^{1-n}ds\le \alpha^{2-n}$
  and recalling that $\alpha\le \beta$ we obtain \eqref{eq:alphbe}.
  If $\alpha\ge \beta\ge \alpha/3$, we have in a similar way
  \begin{equation*}
    \textstyle
    |V(\alpha)|\le
    \int_{\beta}^{\alpha}|V'|+|V(\beta)|
    \le
    \left(
      \int_{\beta}^{\alpha}s^{n-1}|V'|^{2}
    \right)^{\frac 12}
    (\int_{\beta}^{\alpha}s^{1-n})^{\frac 12}
    +|V(\beta)|.
  \end{equation*}
  Now 
  $\int_{\beta}^{\alpha}s^{1-n}ds\le \beta^{2-n}$
  and we get
  \begin{equation*}
    \textstyle
    \beta^{\frac n2-1}|V(\alpha)|
    \le 
    \left(
      \int_{I}s^{n-1}|V'|^{2}
    \right)^{\frac 12}
    + \beta^{\frac n2-1}|V(\beta)|
  \end{equation*}
  and recalling that $\beta\ge \alpha/3$ we obtain again
  \eqref{eq:alphbe}.

  Next, split $I$ in thirds $I=I_{1}\cup I_{2}\cup I_{3}$
  (with $I_{1}$ at the left and $I_{3}$ at the right).
  If $\alpha \in I_{1}\cup I_{2}$, pick $\beta\in I_{3}$ arbitrary
  and apply \eqref{eq:alphbe} to get
  \begin{equation*}
    \textstyle
    \alpha^{\frac n2-1}|V(\alpha)|
    \lesssim 
    \left(
      \int_{I}s^{n-1}|V'|^{2}
    \right)^{\frac 12}
    +\inf_{\beta\in I_{3}} \beta^{\frac n2-1}|V(\beta)|.
  \end{equation*}
  Since
  \begin{equation*}
    \textstyle
    (\int_{I_{3}}s^{n-1}|V|^{2})^{\frac 12}
    \ge
    \inf _{I_{3}}\beta^{\frac n2-1}|V(\beta)|
    \cdot
    (\int_{I_{3}}sds)^{\frac 12}
    \gtrsim
    |I|
    \inf _{I_{3}}\beta^{\frac n2-1}|V(\beta)|
  \end{equation*}
  (indeed 
  $\int_{a}^{b}sds=\frac{b^{2}-a^{2}}{2}\ge \frac{(b-a)^{2}}{2}$)
  our claim \eqref{eq:inqI} is proved for the points 
  in $I_{1}\cup I_{2}$.
  On the other hand, if $\alpha\in I_{3}$, we pick
  $\beta\in I_{2}$ arbitrary, we apply \eqref{eq:alphbe},
  and we get
  \begin{equation*}
    \textstyle
    \alpha^{\frac n2-1}|V(\alpha)|
    \lesssim 
    \left(
      \int_{I}s^{n-1}|V'|^{2}
    \right)^{\frac 12}
    +\inf_{\beta\in I_{2}} \beta^{\frac n2-1}|V(\beta)|
  \end{equation*}
  and the same argument gives again \eqref{eq:inqI}.
\end{proof}

We apply \eqref{eq:inqI} to $U(T,\alpha)$ at a fixed
$T\in(\delta/3,\pi)$ on the interval 
\begin{equation*}
  \textstyle
  I=[\Gamma(T), \pi-T+\delta/4]
  \subseteq(0,K]=(0,\pi-\frac{\delta}{12}];
\end{equation*}
note that $|I|\ge \delta/4$.
We get
\begin{equation*}
  \sup_{I} \alpha^{\frac n2-1}|U(T,\alpha)|
  \lesssim
  \textstyle
  \delta^{-1}
  \left[\int_{\Gamma(T)}^{\pi-T+\delta/4}\alpha^{n-1}
  (|\partial_{\alpha}U|^{2}+|U|^{2})d \alpha\right]^{\frac 12}.
\end{equation*}
For $\alpha\in[0,\pi-\frac \delta{12}]$ we have
\begin{equation*}
  \sin \alpha\ge
  \frac{\sin(\pi-\delta/12)}{\pi-\delta/12}\cdot\alpha
  \quad\implies\quad
  \alpha \lesssim \delta^{-1}\sin \alpha,
\end{equation*}
thus substituting in the previous inequality we get,
for all $\alpha\in[\Gamma(T),\pi-T+\delta/4]$,
\begin{equation*}
  (\sin\alpha)^{\frac n2-1}|U(T,\alpha)|
  \lesssim
  \textstyle
  \delta^{-\frac{n+1}{2}}
  \left[\int_{\Gamma(T)}^{\pi-T+\delta/4}(\sin\alpha)^{n-1}
  (|\partial_{\alpha}U|^{2}+|U|^{2})d \alpha\right]^{\frac 12}.
\end{equation*}
Recalling the definition of $F(T)$ and \eqref{eq:enest2}
we conclude
\begin{equation}\label{eq:finalest}
  (\sin\alpha)^{\frac n2-1}|U(T,\alpha)|
  \lesssim
  \delta^{-\frac{n+1}{2}} E(0)^{1/2}.
\end{equation}

We now convert \eqref{eq:finalest} into an estimate for $u(t,x)$.
Since $\sin \alpha=\omega r$, we obtain
\begin{equation*}
  (\sin\alpha)^{\frac n2-1}|U(T,\alpha)|
  \gtrsim
  (\omega r)^{\frac n2-1}\omega^{\frac{1-n}{2}}|u(t,x)|=
  r^{\frac n2-1}\bra{t+r}^{\frac 12}\bra{t-r}^{\frac 12}
  |u(t,x)|.
\end{equation*}
On the other hand, a change of variable shows that
\begin{equation*}
  \|U(0,\cdot)\|_{L^{2}(\widetilde{\mathbb{S}^{n}})}
  =
  \sqrt{2}\|\bra{r}^{-1}u_{0}\|_{L^{2}(\Omega)}
\end{equation*}
where $\widetilde{\mathbb{S}^{n}}$ denotes the image of 
$\{t=0\}\times\Omega$
via the Penrose transform,
\begin{equation*}
  \|U(0,\cdot)\|
  _{L^{p+1}(\widetilde{\mathbb{S}^{n}})}
  =
  \sqrt{2}^{1-n \frac{p-1}{p+1}}
  \|\bra{r}^{n \frac{p-1}{p+1}-1}u_{0}\|_{L^{p}(\Omega)},
\end{equation*}
\begin{equation*}
  \|\partial_{T}U(0,\cdot)\|_{L^{2}(\widetilde{\mathbb{S}^{n}})}
  =
  \sqrt{2}^{-1}\|\bra{r}u_{1}\|_{L^{2}(\Omega)}
\end{equation*}
and
\begin{equation*}
  \|\partial_{\alpha}U(0,\cdot)\|_{L^{2}(\widetilde{\mathbb{S}^{n}})}
  \lesssim
  \|\bra{r}^{-1}u_{0}\|_{L^{2}(\Omega)}
  +
  \|\bra{r}\nabla_{x} u_{0}\|_{L^{2}(\Omega)}.
\end{equation*}
This gives the estimate
\begin{equation}\label{eq:estF0}
  E(0)^{1/2}
  \lesssim
  \|\bra{x}^{-1}u_{0}\|_{L^{2}}
  +
  \|\bra{x}^{n \frac{p-1}{p+1}-1}u_{0}\|_{L^{p+1}}^{\frac{p+1}{2}}
  +
  \|\bra{x}(|\nabla u_{0}|+|u_{1}|)\|_{L^{2}}
\end{equation}
and recalling \eqref{eq:finalest} (and the dependence of
$\delta$ in \eqref{eq:depd}), we conclude
the proof of Theorem \ref{the:decayrad}.

\subsection{Proof of Theorem \ref{the:regul}}
Let $u$ be the solution given by 
Proposition \ref{pro:radial} and Theorem \ref{the:decayrad},
and $v$ be the local solution given by
Lemma \ref{lem:localfu}, maximally extended to a lifespan
$[0,T^{*})$. Since the data are radial, $v$ is also radial,
and by the uniqueness part of Proposition \ref{pro:radial} we see
that $u \equiv v$. 
In particular, $v$ is bounded as 
$t \uparrow T^{*}$, hence $T^{*}=+\infty$. This proves
the regularity of the radial solution.

It remains to prove the uniform bounds \eqref{eq:boundsHk}.
For the $L^{2}$ norm we have
\begin{equation*}
  \|u(t,\cdot)\|_{L^{2}}
  \le
  \|u_{0}\|_{L^{2}}+\|\Lambda^{-1} u_{1}\|_{L^{2}}
  +
  \textstyle
  \int_{0}^{t}\||u|^{p}\|_{L^{\frac{2n}{n+2}}}ds.
\end{equation*}
Then we can write, if $n\ge3$,
\begin{equation*}
  \||u|^{p}\|_{L^{\frac{2n}{n+2}}}\le
  \|u\|_{L^{\infty}}^{p-\frac{n+2}{n-2}}
  \|u\|_{L^{\frac{2n}{n-2}}}^{\frac{n+2}{n-2}}
  \lesssim
  \|u\|_{L^{\infty}}^{p-\frac{n+2}{n-2}}
  \|\nabla_{x}u\|_{L^{2}(\Omega)}^{\frac{n+2}{n-2}}
\end{equation*}
by H\"{o}lder and Sobolev embedding,
and we note the consequence of \eqref{eq:decaysol}
(also valid only if $n\ge3$)
\begin{equation}\label{eq:tm1}
  \|u\|_{L^{\infty}}\lesssim
  C(\|(u_{0},u_{1})\|_{M})\bra{t}^{-1}
\end{equation}
and the conservation of energy \eqref{eq:consE}. Summing up, we
obtain
\begin{equation*}
  \|u(t,\cdot)\|_{L^{2}}
  \lesssim
  \|u_{0}\|_{L^{2}}+\|u_{1}\|_{L^{\frac{2n}{n+2}}}
  +
  \textstyle
  \|(u_{0},u_{1})\|_{M}^{p-\frac{n+2}{n-2}}
  E(0)^{\frac{n+2}{2(n-2)}}
  \int_{0}^{t}\bra{s}^{-p+\frac{n+2}{n-2}}ds
\end{equation*}
where the integral converges since 
$p>1+\frac{n+2}{n-2}=\frac{2n}{n-2}$. Noting that
\begin{equation*}
  E(0)\le 
  C(\|(u_{0},u_{1})\|_{M}),
  \qquad
  \|u_{1}\|_{L^{\frac{2n}{n+2}}}
  \lesssim\|\bra{x}u_{1}\|_{L^{2}}
  \lesssim\|(u_{0},u_{1})\|_{M},
\end{equation*}
we get the uniform bound
\begin{equation}\label{eq:L2estu}
  \|u(t,\cdot)\|_{L^{2}}\le C(\|(u_{0},u_{1})\|_{M}).
\end{equation}

Estimate \eqref{eq:boundsHk} for $k=1$ is a consequence of the
energy conservation $E(t)=E(0)$ and of \eqref{eq:L2estu}.
For $k>1$ we proceed by induction on $k$.
By the energy estimate \eqref{eq:EnEst} we have
\begin{equation*}
  \textstyle
  \|\nabla_{t,x} u\|_{Y^{\infty,2;k}_{T}}
  \lesssim
  \|u_{0}\|_{H^{k+1}}
  +
  \|u_{1}\|_{H^{k}}
  +
  \||u|^{p}\|_{Y^{1,2;k}_{T}}.
\end{equation*}
We apply Gagliardo--Nirenberg estimates to estimate the last term.
Handling separately the highest order terms 
$\sim\|u\|_{L^{\infty}}^{p-1}\sum_{j\le k}
  \|\partial^{j}_{t}u\|_{H^{k-j}}$, we get
\begin{equation*}
\begin{split}
  \||u|^{p}\|_{Y^{1,2;k}_{T}}=
  &
  \textstyle
  \sum\limits_{j+|\alpha|\le k}
  \int_{0}^{T}\|\partial^{j}_{t}\partial^{\alpha}_{x}
      |u|^{p}\|_{L^{2}}dt
  \\
  \lesssim
  & \textstyle
  \int_{0}^{T}
    (\|u\|_{L^{\infty}}^{p-1}+\|u\|_{L^{\infty}}^{p-k})
    (1+\sum\limits_{\ell\le k-1}
    \|\partial^{\ell}_{t}u\|_{L^{\infty}})^{k}
    dt
    \cdot
    \|u\|_{Y^{\infty,2;k}_{T}}
\end{split}
\end{equation*}
provided $p>k+1$.
Since $u$ is radial in $x$ we have
\begin{equation*}
  \textstyle
  \sum\limits_{\ell\le k-1}
  \|\partial^{\ell}_{t}u\|_{L^{\infty}}
  \lesssim
  \|u\|_{Y^{\infty,2;k}_{T}}
\end{equation*}
which is bounded by the induction hypothesis; on the other hand
$\|u\|_{L^{\infty}}^{p-k}$ is integrable since $p>k+1$ by
\eqref{eq:tm1}. Thus the right hand side is finite and 
this proves \eqref{eq:boundsHk}.

\section{Global quasiradial solutions}\label{sec:quasiradial}

This section is devoted to the proof
of Theorem \ref{the:quasiradial}.

Let $u$ be the global radial solution with initial data
$(u_{0},u_{1})$ given by Proposition \ref{pro:radial} and 
Theorem \ref{the:regul}, and let $v$ be the local
solution with data $(v_{0},v_{1})$
given by Lemma \ref{lem:localfu}.
Denote by $w=v-u$ the difference of the two solutions, which
satisfies the equation
\begin{equation}\label{eq:diffeq}
  \square w+|u+w|^{p-1}(u+w)-|u|^{p-1}u=0.
\end{equation}
We can write
\begin{equation*}
  |u+w|^{p-1}(u+w)-|u|^{p-1}u=
  p|u|^{p-1}w-w^{2}\cdot F[u,w]
\end{equation*}
where
\begin{equation}\label{eq:expF}
  \textstyle
  F[u,w]=-
  p(p-1)
  \int_{0}^{1}|u+\sigma w|^{p-3}(u+\sigma w)(1-\sigma)d \sigma
\end{equation}
so that $w$ satisfies
\begin{equation*}
  \square w+V(t,x)w=w^{2}\cdot F[u,w],
  \qquad
  V(t,x)=p|u|^{p-1}.
\end{equation*}
For $j\le p-3$ we can write
\begin{equation}\label{eq:derV}
  \|\partial^{j}_{t} V\|_{W^{n,1}} \lesssim
  \textstyle
  \sum_{\mu=0}^{j}\|u\|_{L^{\infty}}^{p-3-\mu}
  \|u\|_{L^{2}}^{2}
  \|\partial^{j_{1}}_{t}u\|_{W^{n,\infty}}
  \dots
  \|\partial^{j_{\mu}}_{t}u\|_{W^{n,\infty}}
\end{equation}
where $j_{1}+\dots +j_{\mu}=j$.
By Sobolev embedding and energy estimates \eqref{eq:boundsHk}
we have
\begin{equation*}
  \textstyle
  \|\partial^{j_{\mu}}_{t}u\|_{W^{n,\infty}}
  \lesssim
  \|u\|_{Y^{\infty,2;N}}
  \le
  C
  \Bigl(\|(u_{0},u_{1})\|_{M},
    \|u_{0}\|_{H^{N}},\|u_{1}\|_{H^{N-1}}
  \Bigr)
\end{equation*}
provided $N>n+  \frac n2 +j$,
and this implies, recalling \eqref{eq:tm1},
\begin{equation}\label{eq:thirdV}
  \|\partial^{j}_{t} V(t)\|_{W^{n,1}} \lesssim
  \bra{t}^{j+3-p}
\end{equation}
provided
\begin{equation}\label{eq:conddat}
  \|(u_{0},u_{1})\|_{M}+
      \|u_{0}\|_{H^{N}}+\|u_{1}\|_{H^{N-1}}<\infty,
  \qquad
  \textstyle
  N>\frac 32n+j.
\end{equation}
Now let $m\ge1$ to be chosen and assume $u_{0},u_{1}$
satisfy \eqref{eq:conddat} with
\begin{equation*}
  \textstyle
  N> m+\frac 32n
\end{equation*}
while $w_{0},w_{1}$ satisfy
the compatibility conditions of order $m$.
We see that if $p>m+4$ the assumptions of
Proposition \ref{pro:enerstr} are satisfied and the estimates
\eqref{eq:L2enV}, \eqref{eq:EnEstV} are valid.

Consider now the problem
\begin{equation}\label{eq:pbV}
  \square w+V(t,x)w=w^{2}\cdot F[u,w],
  \qquad
  w(0,x)=w_{0},
  \quad
  w_{t}(0,x)=w_{1},
  \quad
  w(t,\cdot)\vert_{\partial \Omega}=0
\end{equation}
where $w_{j}=v_{j}-u_{j}$, $j=0,1$, and $V(t,x)=p|u|^{p-1}$.
Define the spacetime norm
\begin{equation*}
  \textstyle
  M(T)=
  \|w\|_{Y^{\infty,2;m+1}_{T}}+
  \|w\|_{Y^{q,r;m}_{T}}
\end{equation*}
where the couple $(q,r)$, satisfying \eqref{eq:strsubadm},
will be precised below.
Our goal is to prove the a priori estimate
\begin{equation}\label{eq:claim}
  M(T)\lesssim \epsilon+M(T)^{2}+M(T)^{p}
\end{equation}
where $\epsilon$ is a suitable norm of the data $(w_{0},w_{1})$.

Applying \eqref{eq:L2enV}, \eqref{eq:EnEstV} to 
\eqref{eq:pbV} we get
\begin{equation}\label{eq:basicM}
  M(T)\lesssim 
  \|w_{0}\|_{H^{m+1}}+
  \|w_{1}\|_{H^{m}\cap L^{\frac{2n}{n+2}}}+
  \|w^{2}F[u,w]\|
  _{Y^{q,r;m}_{T}\cap L^{1}_{I}L^{\frac{2n}{n+2}}}
\end{equation}
We must estimate the last term. We begin by noticing that
\begin{equation*}
  \|w^{2}F[u,w]\|_{L^{\frac{2n}{n+2}}}
  \lesssim
  \|w\|_{L^{\frac{4n}{n+2}}}^{2}
  \left(
    \|u\|_{L^{\infty}}^{p-2}+\|w\|_{L^{\infty}}^{p-2}
  \right)
\end{equation*}
and if we assume $m+1>\frac n2$ so that
$\|w\|_{L^{\infty}}\lesssim\|w\|_{H^{m+1}}$, we get
\begin{equation*}
  \|w(t)\|_{L^{\frac{4n}{n+2}}}^{2}
  \le
  \|w\|_{L^{2}}^{\frac{n+2}{n}}
  \|w\|_{L^{\infty}}^{\frac{n-2}{n}}
  \lesssim\|w\|_{H^{m+1}}^{2}
  \le M(T)^{2}
\end{equation*}
for $0\le t\le T$. Since $|u(t,\cdot)|\lesssim \bra{t}^{-1}$
(recall \eqref{eq:tm1}), we have
\begin{equation*}
  \|w^{2}F[u,w]\|_{L^{\frac{2n}{n+2}}}
  \lesssim
  M(T)^{2}(\bra{t}^{2-p}+
  \bra{t}^{(2-p)\frac{n-1}{2}}M(T)^{p-2})
\end{equation*}
and we conclude
\begin{equation}\label{eq:firstM}
  \|w^{2}F[u,w]\|_{L^{1}_{T} L^{\frac{2n}{n+2}}}ds
  \lesssim
  M(T)^{2}+M(T)^{p}
\end{equation}
provided $p>3$.
We next estimate
\begin{equation*}
  \textstyle
  \|w^{2}F[u,w]\|_{Y^{1,2;m}_{T}}=
  \sum_{j+|\alpha|\le m}
  \|\partial^{j}_{t}\partial^{\alpha}_{x}(w^{2}F[u,w])\|
  _{L^{1}_{T}L^{2}}.
\end{equation*}
Recalling the expression of $F$ in \eqref{eq:expF}, we
may expand the derivative as a finite sum
\begin{equation*}
  \textstyle
  \partial^{j}_{t}\partial^{\alpha}_{x}(w^{2}F[u,w])=
  \sum\int_{0}^{1}
  G \cdot W_{1}W_{2} U_{1}\cdot \dots \cdot U_{\nu}d \sigma
\end{equation*}
where $0\le \nu\le m$ and, assuming $p\ge m+2$,
\begin{itemize}
  \item 
  $W_{1}=\partial^{h_{1}}_{t}\partial^{\alpha_{1}}_{x}w$
  and
  $W_{2}=\partial^{h_{2}}_{t}\partial^{\alpha_{2}}_{x}w$
  \item $U_{k}=
    \partial^{j_{k}}_{t}\partial^{\beta_{k}}_{x}(u+\sigma w)$
  \item $h_{1}+h_{2}+j_{1}+\dots +j_{\nu}=j$,
  $\alpha_{1}+\alpha_{2}+\beta_{1}+\dots +\beta_{\nu}=\alpha$,
  and $j+|\alpha|\le m$
  \item it is not restrictive to assume that
  $j_{\nu}+|\beta_{\nu}|\ge j_{k}+|\beta_{k}|$ for
  all $k$ and that $h_{2}+|\alpha_{2}|\ge h_{1}+|\alpha_{1}|$
  \item 
  $G$ satisfies $|G|\lesssim|u+\sigma w|^{p-\nu-2}$
  so that 
  $\|G\|_{L^{\infty}}\lesssim
    (\|u\|_{L^{\infty}}+\|w\|_{L^{\infty}})^{p-\nu-2}$.
\end{itemize}
We take the $L^{2}$ norm in $x$ of each product,
and we estimate it with
the $L^{2}$ norm of the derivative
of highest order, times the $L^{\infty}$ norm of the remaining
factors; it may happen that the highest order derivative falls
on $w$ or on $u+\sigma w$. Thus we need to distinguish two cases:

(1) The highest order derivative is on $u+\sigma w$, 
that is to say
$j_{\nu}+|\beta_{\nu}|\ge h_{2}+|\alpha_{2}|$.
Then we estimate
\begin{equation*}
  \|GW_{1}W_{2}U_{1}\dots U_{\nu}\|_{L^{2}}
  \le
  \|G\|_{L^{\infty}}
  \|W_{1}\|_{L^{\infty}}
  \|W_{2}\|_{L^{\infty}}
  \|U_{1}\|_{L^{\infty}}
  \dots
  \|U_{\nu-1}\|_{L^{\infty}}
  \|U_{\nu}\|_{L^{2}}.
\end{equation*}
We have for $0\le t\le T$
\begin{equation*}
  \|U_{\nu}\|_{L^{2}}\lesssim M(T).
\end{equation*}
Moreover for $i<\nu$ we have
$j_{i}+|\alpha_{i}|\le \frac m2$,
hence if we assume $m>n-2$ we have by Sobolev embedding
\begin{equation*}
  \|U_{i}\|_{L^{\infty}}
  \lesssim
  \|u\|_{Y^{\infty,2;m+1}_{T}}+
  \|w\|_{Y^{\infty,2;m+1}_{T}}
  \le C_{m+1}+M(T),
\end{equation*}
where in the last step we used the energy estimates
\eqref{eq:boundsHk} for $u$, with
\begin{equation*}
  C_{m+1}=C(\|(u_{0},u_{1})\|_{M},\|u_{0}\|_{H^{m+1}},
  \|u_{1}\|_{H^{m}}).
\end{equation*}
In a similar way
\begin{equation*}
  \|W_{i}\|_{L^{\infty}}\lesssim
  M(T),
  \qquad i=1,2
\end{equation*}
Finally, by \eqref{eq:tm1} and Sobolev embedding
\begin{equation*}
  \|G\|_{L^{\infty}}
  \lesssim
  (\|u\|_{L^{\infty}}+ \|w\|_{L^{\infty}})^{p-\nu-2}
  \lesssim
  (\bra{t}^{-1}+ \|w\|_{W^{m,r}})^{q}
  (1+ M(T))^{p-\nu-q-2}
\end{equation*}
provided
\begin{equation*}
  \textstyle
  m>\frac nr
  \quad\text{and}\quad 
  p\ge m+q+2.
\end{equation*}
Summing up we have, for $m>\max\{n-2,n/r\}$, $p\ge m+q+2$,
$0\le t\le T$
\begin{equation*}
  \|GW_{1}W_{2}U_{1}\dots U_{\nu}\|_{L^{2}}
  \lesssim
  (\bra{t}^{-1}+ \|w\|_{W^{m,r}})^{q}
  (1+ M(T))^{p-\nu-q-2}
  M(T)^{2}
  (C_{m+1}+M(T))^{\nu}.
\end{equation*}
We now integrate in $t$ on $[0,T]$ to get
\begin{equation*}
  \textstyle
  \|GW_{1}W_{2}U_{1}\dots U_{\nu}\|_{L^{1}_{T}L^{2}}
  \lesssim
  M(T)^{2}(1+M(T))^{p-q-2}
  \int_{0}^{T}(\bra{t}^{-q}+\|w\|_{W^{m,r}}^{q})dt
\end{equation*}
\begin{equation}\label{eq:finalM}
  \lesssim
  M(T)^{2}(1+M(T))^{p-2}
\end{equation}
with an implicit constant depending on
$\|(u_{0},u_{1})\|_{M}$, $\|u_{0}\|_{H^{m+1}}$, 
$\|u_{1}\|_{H^{m}}$.

(2) The highest order derivative is on $w$, 
that is to say
$j_{\nu}+|\beta_{\nu}|\le h_{2}+|\alpha_{2}|$.
In this case we estimate
\begin{equation*}
  \|GW_{1}W_{2}U_{1}\dots U_{\nu}\|_{L^{2}}
  \le
  \|G\|_{L^{\infty}}
  \|W_{1}\|_{L^{\infty}}
  \|W_{2}\|_{L^{2}}
  \|U_{1}\|_{L^{\infty}}
  \dots
  \|U_{\nu-1}\|_{L^{\infty}}
  \|U_{\nu}\|_{L^{\infty}}.
\end{equation*}
Proceeding in a similar way, we obtain again
\eqref{eq:finalM}.

Summing up, recalling \eqref{eq:basicM}, we have proved
\begin{equation}\label{eq:finalMbis}
  M(T)\lesssim 
  \|w_{0}\|_{H^{m+1}}+
  \|w_{1}\|_{H^{m}\cap L^{\frac{2n}{n+2}}}+
  M(T)^{2}+M(T)^{p}.
\end{equation}
The implicit constant depends on
$\|(u_{0},u_{1})\|_{M}+\|u_{0}\|_{H^{m+1}}+\|u_{1}\|_{H^{m}}$.
The conditions on the parameters are
\begin{equation*}
  \textstyle
  p\ge m+q+2,
  \qquad
  m> \frac nr,
  \qquad
  m>n-2
\end{equation*}
and $(q,r)$ satisfy \eqref{eq:strsubadm}, that is to say
\begin{equation*}
  \textstyle
  \delta\in(0,1),
  \qquad
  p\ge m+\frac 2 \delta+2,
  \qquad
  m>n-2,
  \qquad
  m>\frac n2-1-\frac \delta2.
\end{equation*}
All the constraints are satisfied if
\begin{equation}\label{eq:cho}
  m=n-1,
  \qquad
  p>n+3
\end{equation}
with $\delta<1$ chosen accordingly.
Recall also that in order to apply
Proposition \ref{pro:enerstr} we assumed $p>m+4=n+3$ and
that $u_{0},u_{1}$ are in $H^{N}\times H^{N-1}$ with
compatibility conditions of order $N$, where
$N>m+\frac 32n=\frac 52n-1$.

To conclude the proof it is then sufficient to apply
a standard continuation argument; if the norm
\begin{equation*}
  \|w_{0}\|_{H^{m+1}}+
  \|w_{1}\|_{H^{m}\cap L^{\frac{2n}{n+2}}}
\end{equation*}
of the initial data is sufficiently small with respect
to the hidden constant
\begin{equation*}
  \|(u_{0},u_{1})\|_{M}+\|u_{0}\|_{H^{m+1}}+\|u_{1}\|_{H^{m}}
\end{equation*}
then the quantity $M(T)$ must remain finite as $T\to+\infty$
and global existence is proved.

\section{Weak--strong uniqueness}\label{sec:weak_stro_uniq}

Following \cite{Struwe06}, we prove a more general
stability result which implies Theorem \ref{the:wsuniq}
as a special case. We use the notation
\begin{equation*}
  \textstyle
  E(u)=E(u(t))=\int_{\Omega}
    (\frac{|\partial_{t}u|^{2}+|\nabla_{x}u|^{2}}{2}
      + \frac{|u|^{p+1}}{p+1})dx,
      \qquad
  \Omega=\{|x|>1\}
\end{equation*}
for the energy of a solution $u(t,x)$ at time $t$
of the Cauchy problem
\begin{equation}\label{eq:caugen}
  \square u+|u|^{p-1}u=0,
  \qquad
  u(t,\cdot)\vert_{\partial \Omega}=0.
\end{equation}

\begin{theorem}[]\label{the:stabil}
  Let $I$ be an open interval containing $[0,T]$, $T>0$.
  Let $u,v$ be two distributional solutions to \eqref{eq:caugen} 
  on $I \times \Omega$ such that
  \begin{equation*}
    u\in 
    C(I;H^{2}(\Omega))\cap
    C^{1}(I;H^{1}_{0}(\Omega))\cap
    C^{2}(I;L^{2}(\Omega)),
    \qquad
    u,u_{t}\in
    L^{\infty}(I \times \Omega),
  \end{equation*}
  \begin{equation*}
    v\in 
    C(I;H^{1}_{0}(\Omega))\cap
    C^{1}(I;L^{2}(\Omega))\cap
    L^{\infty}(I;L^{p+1}(\Omega)).
  \end{equation*}
  Assume in addition that $v$ satisfies an energy inequality
  \begin{equation*}
    E(v(t))\le E(v(0)).
  \end{equation*}
  Then the difference $w=v-u$ satisfies the energy estimate
  \begin{equation*}
    E(w(t))\le C e^{Ct}(E(w(0))+\|w(0)\|_{L^{2}(\Omega)}^{2}),
    \qquad
    t\in[0,T]
  \end{equation*}
  where $C$ is a constant depending on
  \begin{equation}\label{eq:const}
    C=C(p,T,\||u|+|u_{t}|\|_{L^{\infty}([0,T]\times \Omega)}).
  \end{equation}
\end{theorem}

\begin{proof}
  We shall need the following easy estimate of the $L^{2}$ norm
  of $w=v-u$:
  \begin{equation}\label{eq:estL2}
    \textstyle
    \|w(t)\|_{L^{2}}^{2}
    \le
    2\|\int_{0}^{t}w_{t}ds\|_{L^{2}}^{2}+2\|w(0)\|_{L^{2}}^{2}
    \le
    2T\int_{0}^{t}E(w(s))ds+2\|w(0)\|_{L^{2}}^{2}.
  \end{equation}
  The difference $w=v-u$ satisfies in the sense of distributions
  \begin{equation*}
    \square w+|u+v|^{p-1}(u+w)-|u|^{p-1}u=0.
  \end{equation*}
  We expand
  \begin{equation*}
    E(v)=E(u)+A(t)+B(t)
  \end{equation*}
  where
  \begin{equation*}
    \textstyle
    A(t)=
    \int \frac{|\partial_{t}w|^{2}+|\nabla_{x}w|^{2}}{2}dx
    +
    \int (\frac{|u+w|^{p+1}-|u|^{p+1}}{p+1}-|u|^{p-1}uw)dx
  \end{equation*}
  \begin{equation*}
    \textstyle
    B(t)=
    \int(u_{t}w_{t}+\nabla u \cdot\nabla w+|u|^{p-1}uw)dx.
  \end{equation*}
  From the energy inequality for $v$, and
  the conservation of energy for the smooth solution $u$,
  we have
  \begin{equation}\label{eq:finalg}
    0\le E(v(0))-E(v(t))=A(0)-A(t)+B(0)-B(t).
  \end{equation}
  We get easily
  \begin{equation*}
    \textstyle
    \frac{|u+w|^{p+1}-|u|^{p+1}}{p+1}-|u|^{p-1}uw
    =
    p\int_{0}^{1}\int_{0}^{\sigma}|u+\tau w|^{p-1}d \tau d \sigma
    |w|^{2}
  \end{equation*}
  \begin{equation*}
    \ge \textstyle
    \frac{2^{2-p}}{p+1}|w|^{p+1}-\frac p2|u|^{p-1}|w|^{2}
  \end{equation*}
  which implies
  \begin{equation*}
    \textstyle
    A(t)\ge 2^{-p}E(w(t))-C\|w\|_{L^{2}}^{2},
    \qquad
    C=\frac p2\|u\|_{L^{\infty}([0,T]\times \Omega)}^{p-1}.
  \end{equation*}
  Recalling \eqref{eq:estL2}, this gives
  \begin{equation*}
    \textstyle
    A(t)\ge 2^{-p}E(w(t))-
    C\int_{0}^{t}E(w(s))ds
    -C\|w(0)\|_{L^{2}}^{2}
  \end{equation*}
  for some $C=C(T,\|u\|_{L^{\infty}([0,T]\times \Omega)})$.
  On the other hand
  \begin{equation*}
    \textstyle
    \frac{|u+w|^{p+1}-|u|^{p+1}}{p+1}-|u|^{p-1}uw
    \le
    C(p,\|u\|_{L^{\infty}})(|w|^{p+1}+|w|^{2})
  \end{equation*}
  which implies
  \begin{equation*}
    A(0)\le CE(w(0))+C\|w(0)\|_{L^{2}}^{2}
  \end{equation*}
  and in conclusion
  \begin{equation}\label{eq:partial1}
    \textstyle
    A(0)-A(t)\le
    -2^{-p}E(w(t))+
        C\int_{0}^{t}E(w(s))ds
        +C\|w(0)\|_{L^{2}}^{2},
  \end{equation}
  with $C=C(p,T,\|u\|_{L^{\infty}})$.

  In order to estimate $B(t)$, we need the following remark.
  Assume the function
  \begin{equation}\label{eq:assW1}
    W(t,x)\in C([0,T];H^{1}_{0}(\Omega))
    \cap C^{1}([0,T];L^{2}(\Omega))
  \end{equation}
  satisfies in $\mathscr{D}'((0,T)\times \Omega)$,
  for some $q\in(1,\infty)$,
  \begin{equation}\label{eq:assW2}
    \square W=F(t,x),
    \qquad
    F\in L^{q}((0,T)\times\Omega).
  \end{equation}
  Then for all $\chi(t)\in C^{\infty}_{c}((0,T))$
  and $U(t,x)\in C^{\infty}_{c}(\mathbb{R}\times \Omega)$
  we have the identity
  \begin{equation}\label{eq:distrib}
    \textstyle
    \iint_{\Omega} \chi'(t)(U_{t}W_{t}+\nabla U \cdot \nabla W)dxdt
    =
    -
    \iint_{\Omega} \chi(t)(W_{t}\square U+U_{t}F)dxdt.
  \end{equation}
  Now assume
  \begin{equation}\label{eq:assU1}
    U\in 
    C(I;H^{2}(\Omega))\cap
    C^{1}(I;H^{1}_{0}(\Omega))\cap
    C^{2}(I;L^{2}(\Omega)),
    \qquad
    U,U_{t}\in
    L^{\infty}(I \times \Omega)
  \end{equation}
  on an open interval $I \supset [0,T]$, so that
  $\square U\in C(I;L^{2}(\Omega))$
  and $U_{t}\in L^{q'}(I\times\Omega)$.
  We can approximate $U$ with a sequence 
  $U_{\epsilon}\in C^{\infty}_{c}(I \times \Omega)$
  in such a way that
  \begin{equation*}
    \partial_{t}U_{\epsilon}\to U_{t},
    \quad
    \nabla U_{\epsilon}\to \nabla U,
    \quad
    \square U_{\epsilon}\to \square U
    \quad\text{in}\quad 
    L^{2}((0,T)\times\Omega)
  \end{equation*}
  and
  \begin{equation*}
    \partial_{t}U_{\epsilon}\to U_{t}
    \quad\text{in}\quad 
    L^{q'}((0,T)\times\Omega)
  \end{equation*}
  as $\epsilon\to0$
  (e.g., extend $U$ as 0 on $\mathbb{R}^{n}\setminus \Omega$,
  apply a radial change of variables 
  $\widetilde{u}(t,x)=u(t,(1-\epsilon)x)$, truncate with a 
  smooth cutoff
  $\psi(|x|-\epsilon^{-1})$ where $\psi=1$ for $|x|\le1$
  and $\psi=0$ for $|x|\ge2$, and finally convolve with
  a delta sequence of the form 
  $\rho_{\epsilon}(x)\sigma_{\epsilon}(t)$).
  Applying \eqref{eq:distrib} to $U_{\epsilon}$ and letting
  $\epsilon\to0$ we obtain that \eqref{eq:distrib} holds
  for any functions $U,W$ satisfying \eqref{eq:assW1}, 
  \eqref{eq:assW2}, \eqref{eq:assU1}
  and any $\chi\in C^{\infty}_{c}((0,T))$.

  Now, consider 
  a sequence of test functions
  $\chi_{k}(t)\in C^{\infty}_{c}((0,T))$, non negative,
  such that $\chi_{k}\uparrow\one{[0,t]}$ pointwise,
  $t\in(0,T]$. We can write
  \begin{equation*}
    \textstyle
    B(t)=B(0)+\lim\iint \chi_{k}'(t)
    (u_{t}w_{t}+\nabla u \cdot \nabla w+|u|^{p-1}uw)dx dt.
  \end{equation*}
  Applying formula \eqref{eq:distrib} with $U=u$, $W=w$,
  $F=|u|^{p-1}u-|u+v|^{p-1}(u+w)$
  (with $q=(p+1)/p$)
  and using the equations for $u,w$ we obtain the identity
  \begin{equation*}
    \textstyle
    B(t)-B(0)=\lim\iint \chi_{k}(s)(|u|^{p-1}uw_{t}-Fu_{t})
    -\lim\iint \chi_{k}(s)\partial_{t}(|u|^{p-1}uw)
  \end{equation*}
  \begin{equation*}
    \textstyle
    =
    \lim\iint \chi_{k}(s)
    (|u+v|^{p-1}(u+w)-|u|^{p-1}u-p|u|^{p-1}w)u_{t}dxds.
  \end{equation*}
  We have
  \begin{equation*}
    \textstyle
    |u+v|^{p-1}(u+w)-|u|^{p-1}u-p|u|^{p-1}w
    =
    p(p-1)\int_{0}^{1}\int_{0}^{\sigma}|u+\tau w|^{p-2}
    (u+\tau w)d \tau d \sigma
    |w|^{2}
  \end{equation*}
  which implies
  \begin{equation*}
    \left|
    |u+v|^{p-1}(u+w)-|u|^{p-1}u-p|u|^{p-1}w
    \right|
    \le
    C(p,\|u\|_{L^{\infty}})(|w|^{p+1}+|w|^{2}).
  \end{equation*}
  Thus we get
  \begin{equation*}
    B(0)- B(t)\le
    \textstyle
    C
    \int_{0}^{t}[E(w(s))+\|w(s)\|_{L^{2}(\Omega)}^{2}]ds,
    \qquad
    C=C(p,\||u|+|u_{t}|\|_{L^{\infty}([0,T]\times \Omega)}).
  \end{equation*}
  Using \eqref{eq:estL2} we conclude
  \begin{equation*}
    \textstyle
    B(0)-B(t)\le C\int_{0}^{t}E(w(s))ds
    +C\|w(0)\|_{L^{2}}^{2}
  \end{equation*}
  for some $C$ as in \eqref{eq:const}.
  Plugging \eqref{eq:partial1}, \eqref{eq:const} in
  \eqref{eq:finalg} we arrive at
  \begin{equation*}
    \textstyle
    E(w(t))\le
    C\int_{0}^{t}E(w(s))ds+CE(w(0))+C\|w(0)\|_{L^{2}}^{2}
  \end{equation*}
  with $C$ as in \eqref{eq:const}, and by Gronwall's Lemma
  we conclude the proof.

\end{proof}


\begin{thebibliography}{10}

\bibitem{Burq03-a}
N. Burq.
\newblock Global {S}trichartz estimates for nontrapping geometries: about an article by {H}. {F}.\ {S}mith and {C}. {D}.\ {S}ogge.
\newblock {\em Comm. Partial Differential Equations}, 9-10:1675--1683, 2003.

\bibitem{Christodoulou86-b}
D.~Christodoulou.
\newblock Global solutions of nonlinear hyperbolic equations for small initial
  data.
\newblock {\em Comm. Pure Appl. Math.}, 39(2):267--282, 1986.


\bibitem{DAnconaLuca12-a}
P.~D'Ancona and R.~Luc\`{a}.
\newblock {Stein-Weiss and Caffarelli-Kohn-Nirenberg inequalities with angular
  integrability.}
\newblock {\em {J. Math. Anal. Appl.}}, 388(2):1061--1079, 2012.

\bibitem{GinibreVelo94-a}
J.~Ginibre and G.~Velo.
\newblock Regularity of solutions of critical and subcritical nonlinear wave
  equations.
\newblock {\em Nonlinear Anal.}, 22(1):1--19, 1994.


\bibitem{Kapitanski94}
L.~Kapitanski.
\newblock Global and unique weak solutions of nonlinear wave equations.
\newblock {\em Math. Res. Lett.}, 1(2):211--223, 1994.





\bibitem{Metcalfe04-a}
J. L. Metcalfe.
\newblock Global {S}trichartz estimates for solutions to the wave equation exterior to a convex obstacle.
\newblock {\em Trans. Amer. Math. Soc.}, 12:4839--4855 (electronic), 2004.


\bibitem{RunstSickel96}
T.~Runst and W.~Sickel.
\newblock {\em Sobolev spaces of fractional order, {N}emytskij operators, and
  nonlinear partial differential equations}, volume~3 of {\em De Gruyter Series
  in Nonlinear Analysis and Applications}.
\newblock Walter de Gruyter \& Co., Berlin, 1996.

\bibitem{ShatahStruwe93-a}
J.~Shatah and M.~Struwe.
\newblock Regularity results for nonlinear wave equations.
\newblock {\em Ann. of Math. (2)}, 138(3):503--518, 1993.

\bibitem{ShatahStruwe94-a}
J.~Shatah and M.~Struwe.
\newblock Well-posedness in the energy space for semilinear wave equations with
  critical growth.
\newblock {\em Internat. Math. Res. Notices}, (7):303ff., approx.\ 7 pp.\
  (electronic), 1994.

\bibitem{ShatahStruwe98-a}
J.~Shatah and M.~Struwe.
\newblock {\em Geometric wave equations}, volume~2 of {\em Courant Lecture
  Notes in Mathematics}.
\newblock New York University Courant Institute of Mathematical Sciences, New
  York, 1998.

\bibitem{ShibataTsutsumi86}
Y.~Shibata and Y.~Tsutsumi.
\newblock On a global existence theorem of small amplitude solutions for
  nonlinear wave equations in an exterior domain.
\newblock {\em Math. Z.}, 191(2):165--199, 1986.

\bibitem{SmithSogge00-a}
H. F. Smith and C. D. Sogge.
\newblock Global {S}trichartz estimates for nontrapping perturbations of the {L}aplacian.
\newblock {\em Comm. Partial Differential Equations}, 11-12:2171-2183, 2000.

\bibitem{Struwe06}
M.~Struwe.
\newblock On uniqueness and stability for supercritical nonlinear wave and
  {S}chr\"{o}dinger equations.
\newblock {\em Int. Math. Res. Not.}, pages Art. ID 76737, 14, 2006.


\end{thebibliography}

\end{document}